\documentclass[a4paper]{scrartcl}

\usepackage[T1]{fontenc}
\usepackage[utf8]{inputenc}

\usepackage[english]{babel}

\usepackage{amsmath,mathtools,amssymb,extarrows,mathrsfs,amsthm}
\usepackage{tikz-cd}
\tikzcdset{scale cd/.style={every label/.append style={scale=#1},
cells={nodes={scale=#1}}}}
\usepackage{xcolor}
\usepackage{float}
\usepackage{hyperref}

\theoremstyle{definition}
\newtheorem{defi}{Definition}[section]

\theoremstyle{plain}
\newtheorem{prop}[defi]{Proposition}

\newtheorem{lemma}[defi]{Lemma}
\newtheorem{corr}[defi]{Corollary}

\newtheorem{thm}[defi]{Theorem}

\theoremstyle{remark}
\newtheorem{rem}[defi]{Remark}

\newtheorem{ex}[defi]{Example}

\makeatletter
\def\blfootnote{\gdef\@thefnmark{}\@footnotetext}
\makeatother

\newcommand{\C}{\mathbb{C}}

\newcommand{\Q}{\mathbb{Q}}
\newcommand{\R}{\mathbb{R}}
\newcommand{\Z}{\mathbb{Z}}
\newcommand{\DD}{\mathrm{D}}

\newcommand{\RR}{\mathrm{R}}
\newcommand{\cA}{\mathcal{A}}
\newcommand{\cB}{\mathcal{B}}
\newcommand{\cC}{\mathcal{C}}

\newcommand{\cF}{\mathcal{F}}
\newcommand{\cG}{\mathcal{G}}

\newcommand{\cL}{\mathcal{L}}

\newcommand{\cO}{\mathcal{O}}
\newcommand{\cP}{\mathcal{P}}
\newcommand{\cR}{\mathcal{R}}
\newcommand{\cS}{\mathcal{S}}

\newcommand{\Mod}[1]{\mathrm{Mod}(#1)}

\newcommand{\Dbk}[1]{\mathrm{D}^\mathrm{b}(k_{#1})}

\newcommand{\DbK}[1]{\mathrm{D}^\mathrm{b}(K_{#1})}
\newcommand{\DbL}[1]{\mathrm{D}^\mathrm{b}(L_{#1})}

\newcommand{\DbRck}[1]{\mathrm{D}^\mathrm{b}_{\R\text{-}\mathrm{c}}(k_{#1})}
\newcommand{\DbRcK}[1]{\mathrm{D}^\mathrm{b}_{\R\text{-}\mathrm{c}}(K_{#1})}
\newcommand{\DbRcL}[1]{\mathrm{D}^\mathrm{b}_{\R\text{-}\mathrm{c}}(L_{#1})}

\newcommand{\DbCck}[1]{\mathrm{D}^\mathrm{b}_{\C\text{-}\mathrm{c}}(k_{#1})}
\newcommand{\DbCcK}[1]{\mathrm{D}^\mathrm{b}_{\C\text{-}\mathrm{c}}(K_{#1})}
\newcommand{\DbCcL}[1]{\mathrm{D}^\mathrm{b}_{\C\text{-}\mathrm{c}}(L_{#1})}
\newcommand{\Vect}[1]{\mathrm{Vect}_{#1}}

\newcommand{\RHom}{\RR\mathcal{H}om}
\newcommand{\Hom}{\mathcal{H}om}
\newcommand{\sHom}{\mathrm{Hom}}

\newcommand{\Pervk}[1]{\mathrm{Perv}(k_{#1})}

\newcommand{\PervK}[1]{\mathrm{Perv}(K_{#1})}
\newcommand{\PervL}[1]{\mathrm{Perv}(L_{#1})}

\newcommand{\Ff}{\mathsf{F}_f} 
	
\newcommand{\GDXfk}{\mathcal{GD}_f(X,k)}
\newcommand{\GDXfK}{\mathcal{GD}_f(X,K)}
\newcommand{\GDXfL}{\mathcal{GD}_f(X,L)}

\newcommand{\NCunf}{\Psi^\mathrm{un}_f}
\newcommand{\VCunf}{\Phi^\mathrm{un}_f}

\newcommand{\iso}{\simeq}

\newcommand{\To}{\longrightarrow}
\newcommand{\ToPO}{\overset{+1}{\longrightarrow}}

\newcommand{\id}{\mathrm{id}}
\newcommand{\supp}{\mathop{\mathrm{supp}}}
\newcommand{\msupp}{\mathop{\mathrm{SS}}}
\newcommand{\op}{\mathrm{op}}

\title{An introduction to field extensions\\and Galois descent for sheaves\\ of vector spaces}
\author{Andreas Hohl}
\date{}

\begin{document}
\maketitle
\vspace{-1cm}

\begin{abstract}
	We study extension of scalars for sheaves of vector spaces, assembling results that follow from well-known statements about vector spaces, but also developing some complements. In particular, we formulate Galois descent in this context, and we also discuss the case of derived categories and perverse sheaves. Most of the results are not new, but our aim is to give an accessible introduction to this subject relying only on techniques from basic sheaf theory. Our proofs also illustrate some applications of results about the structure of constructible and perverse sheaves.\blfootnote{The author's research was funded by the Deutsche Forschungsgemeinschaft (DFG, German Research Foundation), Projektnummer 465657531.}
\end{abstract}

\tableofcontents

\section{Introduction}

Given a field extension $L/K$, it is easy to produce an $L$-vector space $V$ out of a $K$-vector space $W$ by ``extending scalars'' (or ``change of rings''), i.e.\ setting $V\vcentcolon= L\otimes_K W$. The case of vector spaces is, of course, the most basic example, but similar constructions are performed in more advanced contexts: For instance, one can ``upgrade'' a scheme over $K$ to a scheme over $L$. While it is certainly interesting to study such an extension functor itself, one might also wonder if it is possible to describe more precisely its essential image and a possible way of reconstructing an object over $K$ from its associated object over $L$ and some extra data. The last question has nice answers mainly in the case where $L/K$ is a Galois extension, and the machinery behind its solution is commonly referred to as \emph{Galois descent}. In the case of vector spaces (where Galois descent is actually a special case of faithfully flat descent for modules), such questions have, for example, been studied in classical literature such as \cite{Winter}, \cite{Jacobson} \cite{Waterhouse}, \cite{Borel}. We would also like to mention the surveys \cite{Conrad} and \cite{Jahnel}.

If $k$ is a field, $k$-vector spaces are nothing but modules over the constant sheaf $k$ on the one-point space, which is the same as the structure sheaf if we consider the one-point space as $\mathrm{Spec}\, k$. They therefore have two natural generalizations: One can think about $\cO_X$-modules on more general varieties or schemes $X$ over $k$, or one can think about modules over the constant sheaf $k_X$ on more general topological spaces $X$. New questions arise in these contexts, since one can, for instance, ask about compatibility of extension of scalars with operations on sheaves (such as direct and inverse images, duality, etc.). The first viewpoint seems to be widely established (see e.g.\ \cite{Jahnel} for an overview and \cite[Tag 0CDQ]{Stacks}).

In this work, we are going to take the second viewpoint and study extension and descent questions in the case of sheaves of vector spaces and related categories. Most of the basic results from the theory of extension of scalars and Galois descent for vector spaces imply (more or less directly) similar results for sheaves of vector spaces, and indeed our main reference will be classical sheaf theory (see, for example, \cite{KS90}). However, the technical subtleties of certain statements are not always obvious, and it is difficult to keep an overview of the conditions that are required for every single statement: Some assertions hold for arbitrary field extensions and sheaves, others require finite or Galois extensions or certain constructibility assumptions on the sheaves involved. Although the main results in this direction are certainly known to experts or follow from more general frameworks, a canonical reference for the details in the concrete setting of sheaves of vector spaces does not seem to exist. These concepts play, however, a crucial role in theories like that of mixed Hodge modules, where extension of scalars is used for perverse sheaves.

The aim of this work is therefore twofold: Firstly, we want to collect and present the main definitions and statements about extension of scalars and Galois descent for sheaves of vector spaces, together with the arguments needed to deduce them from the well-known statements in linear algebra. Great parts are therefore meant to be rather expository and accessible from an elementary perspective. We do not claim originality for these statements, but we hope that this overview will serve as a useful reference.

Secondly, in the course of this detailed presentation, we establish some complementary technical results: We describe compatibilities between the six Grothendieck operations (and in particular the functor $\RHom$) and extension of scalars in the (derived) category of sheaves of vector spaces. Moreover, we investigate Galois descent for complexes in the derived category of sheaves of vector spaces as well as for perverse sheaves. Again, these results are neither new nor surprising, but all these considerations involve in particular some interesting arguments using results on the structure of $\R$-constructible and perverse sheaves. Indeed, discussions following a former version of this article led to the work \cite{HS23}, whose results we now use here. Even for readers already familiar with the basic ideas, this article may be an enlightening illustration of some ``unusual'' functorialities for $\R$-constructible sheaves and gluing techniques for perverse sheaves.

We originally got interested in this subject during the preparation of our joint article with Davide Barco, Marco Hien and Christian Sevenheck \cite{BHHS22}, where we studied certain differential equations from a topological viewpoint. The basic idea is the following: A Riemann--Hilbert correspondence is an equivalence between certain categories of differential systems and categories of topological objects (such as local systems and perverse sheaves, for example). These topological objects are a priori defined over the field of complex numbers, so one can ask under which conditions such an object ``descends'' to one over a subfield of $\C$, and it turned out that Galois descent serves as a useful technique there. In our common work, we established some Galois descent results for sheaves (and more general objects), some of which we will reformulate and complement in this work. We will also investigate the case of complexes of sheaves.

Although this was our original motivation for studying Galois descent for sheaves of vector spaces, we will not make any reference to the concrete application in loc.~cit.\ here, nor to the more general framework of enhanced ind-sheaves we were studying there. We rather consider this an independent, self-contained and accessible exposition of the subject, providing more details for a broader readership familiar with Galois and sheaf theory. It is, however, one purpose of this article to give constructions that might be imitated in the more general context of enhanced ind-sheaves later. In particular, the method of proof of Galois descent for perverse sheaves can serve as a model for the case of enhanced perverse sheaves.

Let us note that the classical theory of Galois descent for vector spaces can also be adapted to the case of infinite Galois extensions. We will not discuss this case explicitly in this article, although it is certainly interesting to study it also in the context of sheaves.

\paragraph{Outline} After reviewing some basics of sheaf theory in Section~\ref{sec:Sheaves} (mainly to set notation and terminology), we describe the concepts of $G$-structures, extension of scalars, $K$-structures and Galois descent in a quite abstract categorical framework in Section~\ref{sec:Cat}. On the one hand, this allows us to define these notions for multiple categories at the same time, on the other hand, it might also serve as a useful framework for studying them in different examples later. After each definition, we directly give the explicit description in our categories of interest. We address compatibility questions between extension of scalars and the six Grothendieck operations in Section~\ref{sec:Hom}. For some of them, $\R$-constructibility will be required, while others hold without such an assumption.

In Section~\ref{sec:GaloisDescent}, we first describe explicitly Galois descent for sheaves and formulate it as an equivalence of categories, using in particular the results of the previous section to obtain full faithfulness. We then briefly discuss descent in derived categories of sheaves of vector spaces. We do not obtain an equivalence as in the abelian case, but we will give some background and explanations on the problems that arise. Finally, the last subsection is devoted to the study of Galois descent for perverse sheaves. We use a construction of A.\ Beilinson \cite{Bei} -- which we will briefly recall and study in the context of extension of scalars -- to realize perverse sheaves by ``gluing data'' and inductively reduce to the case of sheaves proved before.

\paragraph{Acknowledgements} I would like to thank Claude Sabbah for invaluable discussions, which in particular helped me to work out Galois descent for perverse sheaves using Beilinson's construction. I am also indebted to Pierre Schapira for his interest and comments that helped me to improve a previous version of this article and led to our common work \cite{HS23} about some interesting functorialities for constructible sheaves, which are now used here. I moreover thank Davide Barco, Marco Hien and Christian Sevenheck who, through our common work, motivated me to work out more details on this subject. Finally, I am particularly grateful to Takuro Mochizuki for answering my questions and providing some ideas during the preparation of \cite{BHHS22}, which helped me to better understand the theory.

\section{A very short review of sheaf theory}\label{sec:Sheaves}
We assume the readers to be familiar with the theory of sheaves of vector spaces. We will recall here some basic facts and notations, and refer to the standard literature such as \cite{KS90} or \cite{Dimca} for details.

Although not strictly necessary in all places, we will assume all our topological spaces to be \emph{good}, i.e.\ Hausdorff, locally compact, second countable and of finite cohomological dimension. (This is particularly important for the construction of the functors $f_!$ and $f^!$.)

\paragraph{Presheaves and sheaves} Let $X$ be a topological space and let $\mathrm{Op}(X)$ be the category of open subsets of $X$, where $\sHom(U,V)$ has one element if $U\subseteq V$ and is empty otherwise. Let $k$ be a field.

A presheaf of $k$-vector spaces on $X$ is a functor $\mathrm{Op}(X)^\op\to \Vect{k}$. It is a sheaf if for every open $U\subseteq X$ and every open covering $U=\bigcup_{i\in I} U_i$ the natural sequence
$$0\to \cF(U)\to \prod_{i\in I} \cF(U_i) \rightrightarrows \prod_{i,j\in I} \cF(U_i\cap U_j)$$
is exact. If $\cP$ is a presheaf, its sheafification will be denoted by $\cP^\#$. The sheafification functor is left adjoint to the natural inclusion of sheaves into presheaves.

We denote the Grothendieck abelian category of sheaves of $k$-vector spaces on $X$ by $\Mod{k_X}$, and its bounded derived category by $\Dbk{X}$. There are the six Grothendieck operations $\RR f_*$, $\RR f_!$, $f^{-1}$, $f^!$, $\otimes$ and $\RHom$ (and the underived versions of all these functors except $f^!$). We will sometimes write $\otimes_{k_X}$ or $\RHom_{k_X}$ if we want to emphasize the field.

\paragraph{(Locally) constant sheaves} If $V\in\Vect{k}$, we denote by $V_X\in\Mod{k_X}$ the constant sheaf with stalk $V$ on $X$. It is the sheafification of the constant presheaf $V^\mathrm{pre}_X$ defined by $V^\mathrm{pre}_X(U)=V$ for all $U\in \mathrm{Op}(X)$. If $p_X\colon X\to\{\mathrm{pt}\}$ is the map to the one-point space, we also have $V_X=p_X^{-1}V$ (note that sheaves of vector spaces on a one-point space are the same as vector spaces).
In particular, we have the constant sheaf $k_X$, and if $f\colon X\to Y$ is a morphism of topological spaces, then $f^{-1}k_Y\iso k_X$.

If $\cF\in\Mod{k_X}$ and $Z\subseteq X$ is a locally closed subset with inclusion $j\colon Z\hookrightarrow X$, we write $\cF_Z\vcentcolon= j_!j^{-1}\cF$. This is the restriction of $\cF$ to $Z$, extended again to $X$ by zero outside of $Z$. We will also sometimes denote by $k_Z\in\Mod{k_X}$ the sheaf $j_!k_Z$ if there is no risk of confusion.

We moreover denote the duality functor by $\DD_X\vcentcolon= \RHom(-,\omega_X)$, where $\omega_X=p_X^! k$ is the dualizing complex. We sometimes write $\DD_X^k$ to emphasize the field.

We call a sheaf $\cF\in\Mod{k_X}$ \emph{locally constant} (or a \emph{local system}) if every point $x\in X$ has an open neighbourhood $U\subseteq X$ such that $\cF|_U$ is isomorphic to a constant sheaf $V_X$ for some $V\in\Vect{k}$. It is \emph{locally constant of finite rank} if all the $V$ are finite-dimensional.

\paragraph{Constructibility and perversity} If $X$ is a real analytic manifold, $\cF\in\Mod{k_X}$ is called \emph{$\R$-constructible} if there exists a locally finite covering $X=\bigcup_{\alpha\in A} X_\alpha$ by subanalytic subsets such that $\cF|_{X_\alpha}$ is locally constant of finite rank for every $\alpha$. We denote by $\DbRck{X}$ the full subcategory of $\Dbk{X}$ of complexes with $\R$-constructible cohomologies. We also note that this category is in fact equivalent to the bounded derived category of the category of $\R$-constructible sheaves.

If $X$ is a complex manifold, $\cF\in\Mod{k_X}$ is called \emph{$\C$-constructible} if there exists a locally finite covering $X=\bigcup_{\alpha\in A} X_\alpha$ by $\C$-analytic subsets such that $\cF|_{X_\alpha}$ is locally constant of finite rank for every $\alpha$. We denote by $\DbCck{X}$ the full subcategory of $\Dbk{X}$ of complexes with $\C$-constructible cohomologies.

More intrinsically, for any $\cF\in\Dbk{X}$ ($X$ a differentiable manifold), one can define its microsupport $\msupp(\cF)$, which is a subset of the cotangent bundle $T^*X$. Then, if $X$ is a real analytic (resp.\ complex) manifold, being $\R$-constructible (resp.\ $\C$-constructible) is equivalent to $\msupp(\cF)$ being a closed conic subanalytic (resp.\ $\C$-analytic) Lagrangian subset of $T^*X$. We refer to \cite[Chap.\ V and VIII]{KS90} for details on these notions.

An object $\cF^\bullet\in\DbCck{X}$ is called a \emph{perverse sheaf} if $\dim \supp \mathrm{H}^{-i}(\cF^\bullet)\leq i$ for any $i\in \Z$ (we say that $\cF^\bullet$ satisfies the \emph{support condition}) and $\dim \supp \mathrm{H}^{-i}(\DD_X\cF^\bullet)\leq i$ for any $i\in \Z$ (i.e.\ $\DD_X\cF^\bullet$ also satisfies the support condition). We denote by $\Pervk{X}$ the full subcategory of $\DbCck{X}$ consisting of perverse sheaves. We refer in particular to \cite{BBD} for the theory of perverse sheaves. The category $\Pervk{X}$ is an abelian category. Let us note that, in particular, a perverse sheaf has no nontrivial cohomologies in degrees less than $-\dim_\C X$. (This follows easily from \cite[p.\ 56]{BBD}, for example.)

\paragraph{A remark on operations and sheafification} Before entering the main topic of the article, let us remark the following elementary fact that we are going to use throughout the article:
If $\cP$ is a presheaf and $\cF$ is a sheaf, then the tensor product sheaf $\cP^\#\otimes \cF$ is isomorphic to the sheafification of the (naïve) presheaf tensor product $\cP\overset{\mathrm{pre}}{\otimes}\cF$. This is easy to show using tensor-hom adjunctions for presheaves and sheaves and the universal property of sheafification. It is, however, crucial that $\cF$ is already a sheaf here. In particular, this means that if $L/K$ is a field extension (i.e.\ $L$ is a $K$-module) and $\cF\in\Mod{K_X}$, then the sheaf $L_X\otimes\cF$ is the sheafification of the presheaf $U\mapsto L\otimes_K \cF(U)$.

Let us also note that sheafification does not commute with operations on (pre-)sheaves in general: For example, it is not true in general that $f_*(\cP^\#)\iso (f_*^\mathrm{pre}\cP)^\#$ for a presheaf $\cP$. Indeed, if this were true, we would not have examples as in Remark~\ref{rem:directImage}.

\section{Linear categories and field extensions}\label{sec:Cat}
The general philosophy of Galois descent is the following: Given a Galois extension $L/K$ with Galois group $G$ and an object $F$ over $L$, then the existence of a $G$-structure (i.e.\ a suitable collection of isomorphisms between $F$ and its Galois conjugates, often formulated as a suitable action of the Galois group on $F$) should guarantee the existence of a $K$-structure of $F$, i.e.\ an object over $K$ which is isomorphic to $F$ after extension of scalars. (Even more, a $G$-structure should in some sense determine a particular $K$-structure, since in the case where ``object'' means ``finite-dimensional vector space'', the pure existence of such a structure is not big news.)

We first set up a very general framework for the concepts of $G$- and $K$-structures, motivated by the notions set up in \cite{BHHS22}, but immediately describe these notions explicitly in our categories of interest. Recall that, given a field $k$, a $k$-linear category is a category whose hom spaces are $k$-vector spaces and composition of morphisms is $k$-linear. Note that if $L/K$ is a field extension, any $L$-linear category is automatically also $K$-linear. We will assume any functor between two linear categories to be linear (meaning that the induced map on hom spaces is linear).

\subsection{$G$-structures}

Let $L/K$ be a field extension and denote by $G=\mathrm{Aut}(L/K)$ the group of field automorphisms $g\colon L\to L$ such that $g|_K=\id_K$.

\begin{defi}\label{def:Gconj}
	Let $\cC(L)$ be an $L$-linear category. A \emph{$G$-conjugation} on $\cC(L)$ is a collection of auto-equivalences
	$$\gamma_g=\overline{(\bullet)}^g\colon \cC(L) \overset{\sim}{\To} \cC(L)$$
	(one for each $g\in G$) and natural isomorphisms $$I_{g,h}\colon \gamma_h\circ \gamma_g \overset{\sim}{\To} \gamma_{gh}$$ for any $g,h\in G$ such that for any $g,h,k\in G$ the following diagram is commutative:
	$$\begin{tikzcd}
		& \gamma_{hk}\circ\gamma_g\arrow{dr}{I_{g,hk}} \\
		\gamma_k\circ \gamma_h\circ\gamma_g\arrow{ur}{I_{h,k}\circ \gamma_g} \arrow{dr}[swap]{\gamma_k\circ I_{g,h}} & & \gamma_{ghk}\\
		&\gamma_k\circ\gamma_{gh}\arrow{ur}[swap]{I_{gh,k}}
	\end{tikzcd}$$
\end{defi}
Note that this implies in particular that $\gamma_{\id_L}\iso \id_{\cC(L)}$, where $\id_L\in G$ is the identity element of the $G$. In the categories we use, it will actually be \emph{equal} to the identity functor. More precisely, all the $I_{g,h}$ will be equalities rather than isomorphisms.

\begin{ex}\label{ex:GConj}
	Here are the categories we are interested in in this article: The classical example is that of vector spaces, and it easily generalizes to presheaves and sheaves.
	\begin{itemize}
		\item[(a)] Let $\Vect{L}$ be the category of $L$-vector spaces. Then, for an object $V\in\Vect{L}$ and an element $g\in G$, the $L$-vector space $\overline{V}^g$ is defined as follows: As $K$-vector spaces (or sets), we set $\overline{V}^g=V$, and the action of $L$ on $\overline{V}^g$ is given by
		$$\ell \cdot v \vcentcolon = g(\ell)v$$
		for $\ell\in L$ and $v\in \overline{V}^g$, where the right-hand side is the given scalar multiplication on $V$.
		
		Given a morphism $V\to W$ in $\Vect{L}$, then for any $g\in G$ the same set-theoretic map defines an $L$-linear morphism $\overline{f}^g\colon \overline{V}^g\to\overline{W}^g$.
		
		Altogether, this gives a functor
		$$\overline{(\bullet)}^g\colon \Vect{L}\to\Vect{L},$$
		and it is easy to see that the functor $\overline{(\bullet)}^{g^{-1}}$ is a quasi-inverse, hence the above functor is indeed an auto-equivalence.
		
		Moreover, given $g,h\in G$, the identification $\overline{\overline{V}^g}^h = \overline{V}^{gh}$ is immediate by the definition, as is compatibility of these identifications (that is, commutativity of the diagram in the above definition).
		
		\item[(b)] Let $\cC$ be a category and consider the category $\mathrm{Funct}(\cC,\Vect{L})$ of (covariant) functors from $\cC$ to $\Vect{L}$. Then there is a $G$-conjugation on this category given as follows: Let $F\in\mathrm{Funct}(\cC,\Vect{L})$ and $g\in G$, then we define $\overline{F}^g$ by setting
		$$\overline{F}^g(A)\vcentcolon= \overline{F(A)}^g$$
		for any $A\in \cC$. A morphism $A\to B$ in $\cC$ is sent to the morphism $F(A)\to F(B)$, considered as a morphism $\overline{F(A)}^g\to\overline{F(B)}^g$, as remarked in (a). Thus, this clearly defines an element $\overline{F}^g\in\mathrm{Funct}(\cC,\Vect{L})$.
		
		Given a morphism $F\to G$ in $\mathrm{Funct}(\cC,\Vect{L})$, it is equally easy to see that this induces a morphism $\overline{F}^g\to\overline{G}^g$ for any $g\in G$, and hence we obtain an auto-equivalence $\overline{(\bullet)}^g$ of $\mathrm{Funct}(\cC,\Vect{L})$ with the desired compatibilities.
		
		\item[(c)] Consider the category $\Mod{L_X}$ of sheaves of $L$-vector spaces on a topological space $X$. It is a subcategory of $\mathrm{Funct}(\mathrm{Op}(X)^\op,\Vect{L})$, where $\mathrm{Op}(X)$ is the category of open subsets of $X$ (with inclusions as morphisms). Hence, by (b), to a sheaf $\cF\in\Mod{L_X}$ we can a priori associate a presheaf $\overline{\cF}^g\in \mathrm{Funct}(\mathrm{Op}(X)^\op,\Vect{L})$ for any $g\in G$. It is, however, clear that this presheaf is automatically a sheaf: The unique gluing condition required for sheaves is indeed independent of the action of $L$, it can be checked on the level of sheaves of $K$-vector spaces (or even sets), and on this level $\overline{\cF}^g$ and $\cF$ are the same object. We therefore get an auto-equivalence
		$$\overline{(\bullet)}^g\colon \Mod{L_X}\overset{\sim}{\To}\Mod{L_X}$$
		satisfying the required compatibilities. Moreover, since this equivalence is exact, it also induces a functor
		$$\overline{(\bullet)}^g\colon \DbL{X}\overset{\sim}{\To} \DbL{X}$$
		on the level of derived categories, equipping the latter with a $G$-conjugation.
	\end{itemize}
\end{ex}
In the sequel, when we work in one of these categories, we will always use the $G$-conjugations described in Example~\ref{ex:GConj}.

Since we are mainly concerned with sheaves in this note, let us state the main properties of $G$-conjugation for sheaves of vector spaces.
\begin{lemma}[{cf.\ \cite[Lemma 2.1]{BHHS22}}]\label{lemma:compatConj}
	Let $L/K$ be a field extension and $G\vcentcolon=\mathrm{Aut}(L/K)$. Let $g\in G$. Let $f\colon X\to Y$ be a continuous map between topological spaces.
	\begin{itemize}
		\item[(a)] Let $\cF\in\DbL{X}$. Then $\overline{\RR f_* \cF}^g\iso \RR f_*\overline{\cF}^g$ and $\overline{\RR f_! \cF}^g\iso \RR f_!\overline{\cF}^g$.
		\item[(b)] Let $\cG\in\DbL{Y}$. Then $\overline{f^{-1} \cG}^g\iso f^{-1}\overline{\cG}^g$ and $\overline{f^! \cG}^g\iso f^!\overline{\cG}^g$.
		\item[(c)] Let $\cF_1,\cF_2\in\DbL{X}$. Then $\overline{\RHom(\cF_1,\cF_2)}^g\iso \RHom(\overline{\cF_1}^g,\overline{\cF_2}^g)$ and $\overline{\cF_1\otimes\cF_2}^g\iso \overline{\cF_1}^g\otimes \overline{\cF_2}^g$.
		\item[(d)] Let $\cF\in\DbL{X}$. Then $\overline{\DD_X\cF}^g\iso \DD_X\overline{\cF}^g$.
		\item[(e)] If $\cF\in\Mod{L_X}$ is locally constant, so is $\overline{\cF}^g$.
		\item[(f)] If $X$ is a differentiable manifold and $\cF\in\DbL{X}$, then $\msupp(\cF)=\msupp(\overline{\cF}^g)$.
		
		In particular, if $X$ is a real analytic manifold and $\cF\in\DbRcL{X}$ (resp.\ $X$ is a complex manifold and $\cF\in\DbCcL{X}$), then $\overline{\cF}^g\in\DbRcL{X}$ (resp.\ $\overline{\cF}^g\in\DbCcL{X}$).
		\item[(g)] If $X$ is a complex manifold and $\cF\in\PervL{X}$, then $\overline{\cF}^g\in\PervL{X}$.
	\end{itemize}
\end{lemma}
\begin{proof}
	This proof of (a)--(c) is taken from \cite[Lemma 2.1]{BHHS22}.
	
	By the definitions of $G$-conjugation on sheaves and the direct image functor, we have $(\overline{f_*\cF}^g)(U)=\overline{(f_*\cF)(U)}^g=\overline{\cF(f^{-1}(U))}^g=\overline{\cF}^g(f^{-1}(U))=f_*\overline{\cF}^g(U)$, which shows $\overline{f_*\cF}^g\iso f_*\overline{\cF}^g$. Moreover, we get $\overline{f_!\cF}^g\iso f_!\overline{\cF}^g$, since the proper direct image sheaf is a subsheaf of the direct image sheaf, containing the sections on $U$ whose support is proper over $U$, and the notion of proper support does not depend on the action of $L$, so it is the same subsheaf in both cases.
	
	For $L$-vector spaces $V$ and $W$, it is clear that we have an isomorphism of $L$-vector spaces
	$$\overline{\sHom_L(V,W)}^g \iso \sHom_L(\overline{V}^g,\overline{W}^g)$$
	(sending a morphism on the left to the same set-theoretic map on the right).
	Then for any open $U\subseteq X$, we have
	$$\Hom(\cF_1,\cF_2)(U)=\sHom(\cF_1|_U,\cF_2|_U)\subset \prod_{V\subset U \text{ open}} \sHom(\cF_1(V),\cF_2(V))$$
	(namely the subset of families of morphisms compatible with restriction maps of $\cF_1$ and $\cF_2$).
	This implies $\overline{\Hom(\cF_1,\cF_2)}^g\iso \Hom(\overline{\cF_1}^g,\overline{\cF_2}^g)$ since conjugation is an equivalence and hence commutes with products.
	
	Finally, (a) and the first part of (c) follow by deriving functors (noting that conjugation is exact and preserves injectives), and the other statements follow by adjunction, using the bijection (of sets)
	$$\sHom_{\DbL{X}}(\overline{\cF_1}^g,\cF_2)\iso \sHom_{\DbL{X}}(\cF_1,\overline{\cF_2}^{g^{-1}}).$$
	
	For (d), note that
	\begin{align*}
		\DD_X\overline{\cF}^g &= \RHom(\overline{\cF}^g,\omega_X) \iso \RHom(\overline{\cF}^g,p_X^!L)\\
		&\iso \RHom(\overline{\cF}^g,p_X^!\overline{L}^g)\iso \overline{\RHom(\cF,p_X^!L)}^g \iso \overline{\DD_X\cF}^g.
	\end{align*}
	Here, we have used an isomorphism of $L$-vector spaces $L\iso \overline{L}^g$, which is given by $\ell\mapsto g(\ell)$, as well as (b) and (c).
	
	We now prove (e). Let $U\subseteq X$ be open such that $\cF|_U$ is a constant sheaf, i.e.\ $\cF|_U\iso V_U\iso p_U^{-1}V$ for some $L$-vector space $V$ (where $p_U\colon U\to \{\mathrm{pt}\}$ denotes the map to the one-point space). Then clearly $\overline{\cF}^g|_U\iso p_U^{-1}\overline{V}^g$	is a constant sheaf (with stalk $\overline{V}^g$) by (b).
	
	The invariance of the microsupport (f) follows using the fact that $\RR \Gamma_Z\iso \RR{i_Z}_*i_Z^!$ for a locally closed inclusion $i_Z\colon Z\hookrightarrow X$ and noting that the definition of the microsupport (see \cite[Definition~5.1.2]{KS90}) hence only uses the functors from (a) and (b). Since $\R$- and $\C$-constructibility can be defined by a condition on the microsupport only (see \cite[Theorems~8.4.2 and 8.5.5]{KS90}), these conditions are preserved under conjugation.
	
	Finally, to show (g), note first that, by exactness of conjugation and (b), the support of the cohomologies of an object $\cF\in\DbCcL{X}$ does not change under conjugation. Therefore, if $\cF$ satisfies the support condition, so does $\overline{\cF}^g$. Moreover, by (d), the support condition for $\DD_X\cF$ implies that for $\DD_X\overline{\cF}^g$. Consequently, $\overline{\cF}^g$ is perverse if $\cF$ is.
	
\end{proof}

\begin{defi}\label{def:Gstructure}
	Let $\cC(L)$ be an $L$-linear category with $G$-conjugation. Let $F\in \cC(L)$ be an object. A $G$-structure on $F$ is given by a family $(\varphi_g)_{g\in G}$ of isomorphisms $\varphi_g\colon F\overset{\sim}{\To} \overline{F}^g$ such that for any $g,h\in G$ the following diagram is commutative:
	$$\begin{tikzcd}
		F \arrow[bend left]{rrrrdd}{\varphi_{gh}} \arrow{dd}{\varphi_h}\\ \\
		\overline{F}^h\arrow{rr}{\gamma_h(\varphi_g)} && \overline{\overline{F}^g}^h\arrow{rr}{I_{g,h}} && \overline{F}^{gh}
	\end{tikzcd}$$
\end{defi}

\begin{ex}\label{ex:Gstructures}
	Consider again the category $\Vect{L}$ as in Example~\ref{ex:GConj} (a).
	The trivial one-dimensional vector space $L\in \Vect{L}$ has a natural $G$-structure: Any element $g\in G$ defines an isomorphism of $L$-vector spaces $\varphi_g=g\colon L\overset{\sim}{\To} \overline{L}^g$ given by $\ell\mapsto g(\ell)$ and satisfying the compatibilities from Definition~\ref{def:Gstructure}. (We will often just denote it by the letter $g$.) The same works, of course, for any $n$-dimensional vector space of the form $L^n$ for some $n\in\Z_{>0}$, applying $g$ componentwise.
	
	Similarly, in the category $\Mod{L_X}$ of sheaves of $L$-vector spaces on a topological space, one gets a canonical $G$-structure on the constant sheaf $L_X$. Indeed, the constant sheaf $L_X$ is the sheafification of the constant presheaf $L^\mathrm{pre}_X$ defined by $L^\mathrm{pre}_X(U)=L$ for any open $U\subseteq X$. On $L^\mathrm{pre}_X$, we can define the $G$-structure just as in the example of the trivial vector space $L$ above. Then we conclude via Lemma~\ref{lemma:GstrSheafification} below.
\end{ex}

\begin{lemma}\label{lemma:GstrSheafification}
	Assume that $\cP\in\mathrm{Funct}(\mathrm{Op}(X)^\op,\Vect{L})$ is a presheaf with a $G$-structure $(\phi_g)_{g\in G}$. Then its sheafification $\cF \vcentcolon= \cP^\#$ has an induced $G$-structure $(\varphi_g)_{g\in G}$.
\end{lemma}
\begin{proof}
	For any $g\in G$, we are given
	$$\phi_g\colon \cP\overset{\sim}{\To} \overline{\cP}^g,$$
	and this induces, by the universal property of sheafification, isomorphisms
	$$\cP^\#\overset{\sim}{\To} (\overline{\cP}^g)^\#.$$
	It therefore remains to check that $(\overline{\cP}^g)^\#\iso \overline{(\cP^\#)}^g$.
	
	To see this, note that there is a natural homomorphism $\cP\to\cP^\#$ and hence a natural morphism $\overline{\cP}^g\to \overline{(\cP^\#)}^g$. By the universal property of sheafification, this induces a morphism $(\overline{\cP}^g)^\#\to \overline{(\cP^\#)}^g$. We can check on stalks that it is an isomorphism: $G$-conjugation commutes with taking stalks (see Lemma~\ref{lemma:compatConj}(b)), and stalks of a presheaf and its associated sheaf are the same.
\end{proof}

The following is an easy corollary of Lemma~\ref{lemma:compatConj}. We leave it to the reader to verify that the necessary compatibilities are satisfied.
\begin{corr}
	Let $f\colon X\to Y$ be a morphism of topological spaces. Let $\cF,\cF_1,\cF_2\in \DbL{X}$ and $\cG\in\DbL{Y}$ be equipped with $G$-structures. Then $\RR f_*\cF$, $\RR f_!\cF$, $f^{-1}\cG$, $f^!\cG$, $\cF_1\otimes\cF_2$, $\RHom(\cF_1,\cF_2)$, $\DD_X\cF$ and $\mathrm{H}^i(\cF)$ for any $i\in\Z$ are equipped with an induced $G$-structure.
\end{corr}

We now introduce the following category associated to a category with $G$-conjugation. It will serve as the target category for Galois descent statements later.

\begin{defi}\label{def:catG}
	Let $\cC(L)$ be an $L$-linear category with $G$-conjugation. Then we define the category $\cC(L)^G$ as follows:
	\begin{itemize}
		\item An object of $\cC(L)^G$ is a pair $(F,(\varphi_g)_{g\in G})$, where $F\in \cC(L)$ is an object and $(\varphi_g)_{g\in G}$ is a $G$-structure on $F$.
		\item A morphism $(F,(\varphi_g)_{g\in G})\to (F',(\varphi'_g)_{g\in G})$ is a morphism $f\colon F\to F'$ in $\cC(L)$ compatible with the $G$-structures, i.e.\ such that for any $g\in G$ the diagram
		$$\begin{tikzcd}
			F\arrow{r}{f}\arrow{d}{\iso}[swap]{\varphi_g} & F'\arrow{d}{\varphi'_g}[swap]{\iso}\\
			\overline{F}^g \arrow{r}{\overline{f}^g} & \overline{F'}^g
		\end{tikzcd}$$
		commutes.
	\end{itemize}
\end{defi}

\subsection{Extension of scalars}
As before, let $L/K$ be a field extension and set $G\vcentcolon=\mathrm{Aut}(L/K)$.

Extension of scalars (change of fields) is a well-known principle for vector spaces, sheaves or schemes, for example. In these examples, it is formed by taking tensor products or fibre products of the object over $K$ with an object determined by the field $L$. In the other direction, an object over $L$ can naturally be considered as an object over $K$. (Note that these processes are not inverse to each other, but adjoint.) In an abstract setting, one could define such an extension/restriction datum as follows.
\begin{defi}\label{def:ExtResScalars}
	Let $\cC(K)$ be a $K$-linear category and $\cC(L)$ an $L$-linear category equipped with a $G$-conjugation $((\gamma_g)_{g\in G}, (I_{g,h})_{g,h\in G})$.
	
	Consider a functor 
	$$\Phi_{L/K}\colon \cC(K)\To \cC(L)$$
	that factors as
	$$\cC(K) \xlongrightarrow{\Phi_{L/K}^G} \cC(L)^G \To \cC(L),$$
	where the second functor is just the one forgetting the $G$-structure. Consider furthermore a functor
	$$\mathsf{for}_{L/K}\colon \cC(L)\to\cC(K)$$
	together with natural isomorphisms of functors $J_g\colon \mathsf{for}_{L/K}\circ \gamma_g \overset{\sim}{\to}\mathsf{for}_{L/K}$ for any $g\in G$ such that for any $g,h\in G$ the diagram
	$$\begin{tikzcd}		
		\mathsf{for}_{L/K}\circ \gamma_h\circ \gamma_g \arrow{rr}{\mathsf{for}_{L/K}\circ I_{g,h}} \arrow{d}{J_{h}\circ \gamma_g} && \mathsf{for}_{L/K}\circ \gamma_{gh}\arrow{d}{J_{gh}}\\
		\mathsf{for}_{L/K}\circ \gamma_g \arrow{rr}{J_g} && \mathsf{for}_{L/K}
	\end{tikzcd}$$
	commutes.
	
	If $\Phi_{L/K}$ is left adjoint to $\mathsf{for}_{L/K}$, i.e. there are natural isomorphisms
	$$\sHom_{\cC(K)}(Y,\mathsf{for}_{L/K}(X))\iso \sHom_{\cC(L)}(\Phi_{L/K}(Y),X)$$
	for any $X\in\cC(L)$, $Y\in\cC(K)$, then we call $\Phi_{L/K}$ a functor of \emph{extension of scalars} and $\mathsf{for}_{L/K}$ a functor of \emph{restriction of scalars}.
	
	If, in such a situation, one has objects $F\in\cC(L)$ and $A\in\cC(K)$ such that $\Phi_{L/K}(A)\iso F$, we say that $A$ is a \emph{$K$-structure} (or \emph{$K$-lattice}) of $F$.
	
	For $F\in\cC(L)$, we will often write $F^K\vcentcolon=\mathsf{for}_{L/K}(F)\in \cC(K)$.
\end{defi}

Similarly to what we saw above for the $I_{g,h}$ in Definition~\ref{def:Gconj}, the isomorphisms of functors $J_g$ will be equalities in our examples of interest.

\begin{rem}
	If one wants to work with different field extensions of a common base field, one should require some more compatibilities of the functors of extension and restriction of scalars. Concretely, if $K\subset L\subset L'$ is a composition of two field extensions, then the functors $\Phi_{L'/L}\circ \Phi_{L/K}$ and $\Phi_{L'/K}$ should be isomorphic, in a way compatible with the rest of the data as introduced above, and similarly for $\mathsf{for}_{L/K}\circ \mathsf{for}_{L'/L}$ and $\mathsf{for}_{L'/K}$. 
	Even though we do not use them, these conditions are satisfied in all the examples we consider.
\end{rem}

\begin{ex}
		The functor $\Vect{K}\to\Vect{L}$ defined by $W\mapsto L\otimes_K W$ is a functor of extension of scalars, and the functor $\Vect{L}\to\Vect{K}$ sending an $L$-vector space to itself (but only remembering the action of $K$) is a corresponding functor of restriction of scalars. Indeed, for $W\in\Vect{K}$, the object $L\otimes_K W$ carries a natural $G$-structure: It is clear that $\overline{L\otimes_K W}^g=\overline{L}^g\otimes_K W$, so the $G$-structure is just given by the natural one on $L$ (see Example~\ref{ex:Gstructures}). Moreover, it is clear that $V^K=(\overline{V}^g)^K$ for any $g\in G$, since $g|_K=\id_K$ and hence the action of $K$ is the same on each $\overline{V}^g$. The property of adjointness of these two functors is classical.
		
		Similarly, a functor of extension of scalars for sheaves is given by $\Phi_{L/K}\colon \Mod{K_X}\to\Mod{L_X}, \cG\mapsto L_X\otimes_{K_X} \cG$ (with the obvious functor of restriction of scalars, which is defined as above for each space of sections $\cF(U)$): The sheaf $L_X\otimes_{K_X} \cG$ (a priori a tensor product of two $K_X$-modules) is the sheafification of the presheaf $U\mapsto L\otimes_K \cG(U)$, and hence a sheaf of $L$-vector spaces. Again, the natural $G$-structure on $L_X\otimes_{K_X} \cG$ is given by the natural $G$-structure on $L_X$ (see Example~\ref{ex:Gstructures}). As above, it is clear that $(\overline{\cF}^g)^K=\cF^K$ for any $g\in G$. The fact that these two functors are adjoint follows from Lemma~\ref{lemma:ScalarExtAdj} below.
		
		Noting that both $\Phi_{L/K}$ and $\mathsf{for}_{L/K}$ for sheaves of vector spaces are exact, we also obtain extension/restriction of scalars functors between derived categories of sheaves.
\end{ex}
When we work with these categories, we will always consider the functors of extension of scalars from the above example.

The following lemma describes the adjunction between extension and restriction of scalars in the case of sheaves of vector spaces. (This is rather classical, cf.\ e.g.\ \cite[Tag 0088]{Stacks}, and is analogously proved for the case of two sheaves of rings $\cR\to \cS$ on $X$ instead of the field extension $K_X\to L_X$.)
\begin{lemma}\label{lemma:ScalarExtAdj}
	Let $\cG\in\DbK{X}$ and $\cF\in\DbL{X}$. Then there are isomorphisms
	$$\RHom_{K_X}(\cG,\mathsf{for}_{L/K}(\cF)) \iso \mathsf{for}_{L/K}\big(\RHom_{L_X}(L_X\otimes_{K_X} \cG,\cF)\big)$$
	in $\DbK{X}$ and
	$$\sHom_{\DbK{X}}(\cG,\mathsf{for}_{L/K}(\cF)) \iso \sHom_{\DbL{X}}(L_X\otimes_{K_X} \cG,\cF)$$
	as $K$-vector spaces.
\end{lemma}
\begin{proof}
	For vector spaces $W\in\Vect{K}$ and $V\in\Vect{L}$, it is well-known that
	$$\sHom_K(W,\mathsf{for}_{L/K}(V))\iso \sHom_L(L\otimes_K W, V)$$
	(this is an isomorphism of $K$-vector spaces, i.e.\ there is a ``hidden'' functor $\mathsf{for}_{L/K}$ on the right-hand side),
	and the morphism from left to right is given by $L$-linear continuation.
	
	Let now $\cG\in\Mod{K_X}$ and $\cF\in\Mod{L_X}$. Then we get
	$$\sHom_{\Mod{K_X}}(\cG,\mathsf{for}_{L/K}(\cF)) \iso \sHom_{\Mod{L_X}}(L_X\otimes_{K_X} \cG,\cF)$$
	as follows: An element of the left hom space is a compatible collection of $K$-linear morphisms $\cG(U)\to\cF(U)$ for any open $U\subseteq X$. By the statement for vector spaces, it is easy to see that this corresponds to a compatible collection of $L$-linear morphisms $L\otimes_K \cG(U)\to \cF(U)$, which represent a morphism of presheaves from the presheaf given by $U\mapsto L\otimes_K \cG(U)$ to $\cF$. By the universal property of sheafification (and since $\cF$ is already a sheaf), this is equivalent to a morphism $L_X\otimes_{K_X} \cG\to\cF$.
	
	Finally, we obtain an isomorphism
	$$\Hom_{K_X}(\cG,\mathsf{for}_{L/K}(\cF)) \iso \mathsf{for}_{L/K}\big(\Hom_{L_X}(L_X\otimes_{K_X} \cG,\cF)\big)$$
	by defining it on sections over any open $U\subseteq X$, which means that we need a compatible collection of isomorphisms
	$$\sHom_{\Mod{K_X}}(\cG|_U,\mathsf{for}_{L/K}(\cF|_U)) \iso \sHom_{\Mod{L_X}}(L_U\otimes_{K_U} \cG|_U,\cF|_U)$$
	for any such $U$, which is what we have constructed above.
	
	The statement of the lemma now follows by deriving functors (noting that $\mathsf{for}_{L/K}$ is exact and preserves injectives since it has an exact left adjoint by the above, cf.\ e.g.\ \cite[Tag 015Z]{Stacks}) and taking zeroth cohomology and global sections.
\end{proof}

\subsection{Galois descent}\label{subsec:Galois}

Given a functor of extension of scalars as in the previous subsection immediately raises the question about a construction in the other direction: Given an object over $L$, does it admit a $K$-structure? In general, this is certainly not the case, but the idea is that from a given $G$-structure on $L$ we might indeed be able to construct a $K$-structure. The context in which we can expect such a descent statement is mainly that of a Galois extension $L/K$, since in this case the group $G=\mathrm{Aut}(L/K)$ (which is nothing but the Galois group) ``knows enough'' about the subfield $K$. This idea is made more precise by the concept of Galois descent.

\begin{defi}\label{def:GaloisDescent}
	Let $L/K$ be a Galois extension, $\cC(K)$ a $K$-linear category, $\cC(L)$ an $L$-linear category with a $G$-conjugation, and let these categories be equipped with extension and restriction of scalars as in Definition~\ref{def:ExtResScalars}.
	Then we say that \emph{Galois descent} is satisfied in this setting if $\Phi_{L/K}^G\colon \cC(K)\to\cC(L)^G$ is an equivalence of categories.
\end{defi}

Let now $L/K$ be a finite Galois extension with Galois group $G$. In this case, it is well-known that Galois descent is satisfied for vector spaces, see \cite{Conrad} for an exposition (as well as the references therein and in the introduction above). We briefly recall the construction: The functor
$$\Phi_{L/K}\colon \Vect{K}\To \Vect{L},\qquad V\longmapsto L\otimes_K V$$
induces an equivalence between $K$-vector spaces and $L$-vector spaces equipped with a $G$-structure.
The quasi-inverse of $\Phi_{L/K}^G\colon \Vect{K}\to \Vect{L}^G$ can be explicitly described:

Let $V$ be an $L$-vector space equipped with a $G$-structure $(\varphi_g)_{g\in G}$. Recall that we write $V^K\vcentcolon= \mathsf{for}_{L/K}(V)$. Since $V^K=(\overline{V}^g)^K$ for any $g\in G$, the $L$-linear isomorphisms $\varphi_g\colon V\overset{\sim}{\to} \overline{V}^g$ can be interpreted as $K$-linear automorphisms $\varphi_g^K\vcentcolon= \mathsf{for}_{L/K}(\varphi_g)$ of $V^K$. Then one defines the space of invariants of all these automorphisms
\begin{align*}
	V_K&\vcentcolon= \{v\in V^K\mid \varphi_g^K(v)=v\text{ for any $g\in G$}\} \\ &= \ker \left( \prod_{g\in G} (\varphi_g^K-\id_{V^K})\colon V^K \to \prod_{g\in G} V^K \right).
\end{align*}
This is a $K$-sub-vector space of $V$ and one can show (see \cite[Theorem~2.14]{Conrad} and its proof) that the natural morphism $L\otimes_K V_K\to V, \sum_i \ell_i\otimes v_i\mapsto \sum_i \ell_i v_i$ is an isomorphism.

The aim of the rest of this paper is to investigate the extension-of-scalars functor for sheaves of vector spaces, and to study Galois descent in this framework. We will first establish (not only in the case of Galois extensions) some statements about homomorphisms between extensions of scalars. Then, we will illustrate the problems that occur when we want to set up Galois descent in derived categories, and we will give a proof of Galois descent for perverse sheaves.

\section{Scalar extensions and operations on sheaves}\label{sec:Hom}

Let $L/K$ be a field extension. (We do not assume that $L/K$ is finite or Galois in this section, if not explicitly stated.) Let us remark that -- with the exception of Proposition~\ref{prop:morphLificationFin} -- everything in this section will actually not depend on $L$ being a field and works completely analogously if $K$ is a field and $L$ is a $K$-algebra, for example.

In this section, we want to describe the behaviour of extension of scalars in connection with the six Grothendieck operations for sheaves, and in particular spaces of morphisms $L_X\otimes_{K_X}\cG_1\to L_X\otimes_{K_X} \cG_2$ for $\cG_1,\cG_2\in\DbK{X}$.

As a motivation for the latter, consider the following well-known fact from linear algebra: Let $V,W\in\Vect{K}$. If $L/K$ is a finite field extension or $V$ and $W$ are finite-dimensional, then
\begin{equation*}
\sHom_L(L\otimes_K V,L\otimes_K W)\iso L\otimes_K \sHom_K(V,W).
\end{equation*}
This implies that in the case of finite field extensions the functor of extension of scalars for vector spaces is faithful. In the case of arbitrary field extensions, it is still faithful if we restrict to finite-dimensional vector spaces.

We will develop analogous statements about homomorphism spaces in the context of sheaves. Similarly to the above, we impose some additional assumption if we do not require the field extension to be finite. For sheaves, we will replace the above finiteness assumption by $\R$-constructibility.

Similarly, the compatibilities between extension of scalars and direct and inverse images of sheaves will hold in general in the case of finite fied extensions, while some of them require constructibility assumptions in the general case.

\subsection{Functorialities for sheaf operations}
We will recall here very briefly some compatibilities between the six Grothendieck operations on sheaves that we will use. Let $k$ be a field (which is the case considered in this work, but not necessary for the isomorphisms below).

By the standard theory of sheaves, we have the following natural isomorphisms for $\cF\in \Dbk{X}$, $\cG,\cG_1,\cG_2\in\Dbk{Y}$ and a continuous map $f\colon X\to Y$:
\begin{equation}\begin{aligned}\label{eq:standardCompat}
	&\RR f_*\RHom(f^{-1}\cG,\cF)\iso \RHom(\cG,\RR f_*\cF), && f^{-1}(\cG_1\otimes \cG_2) \iso f^{-1}\cG_1 \otimes f^{-1}\cG_2,\\
	&\RR f_!(f^{-1}\cG\otimes \cF)\iso \cG\otimes \RR f_!\cF, && f^!\RHom(\cG_1,\cG_2) \iso \RHom(f^{-1}\cG_1,f^!\cG_2).
\end{aligned}\end{equation}
These can be found in \cite[Propositions 2.6.4, 2.3.5, 2.6.6 and 3.1.13]{KS90}, for example.

For the remaining combinations of direct and inverse images with hom or tensor product, we only have natural \emph{morphisms} in general:
\begin{align}\label{eq:standardMorph}
	&\RR f_*(f^{-1}\cG\otimes \cF)\leftarrow \cG\otimes \RR f_*\cF, && f^{-1}\RHom(\cG_1, \cG_2) \to \RHom(f^{-1}\cG_1, f^{-1}\cG_2),\notag\\
	&\RR f_!\RHom(f^{-1}\cG, \cF)\to \RHom(\cG,\RR f_!\cF), && f^!(\cG_1\otimes \cG_2) \leftarrow f^{-1}\cG_1\otimes f^!\cG_2.
\end{align}
(These follow from \cite[(2.6.21), (2.6.27), (2.6.15) and Proposition~3.1.11]{KS90}.)

In \cite{HS23}, we show that under certain constructibility conditions on the sheaves involved, these natural morphisms are isomorphisms. In particular, if $\cG_1$ is locally constant (of arbitrary rank) and $\cG_2$ is $\R$-constructible, the two morphisms on the right of \eqref{eq:standardMorph} (those concerning the inverse images) are isomorphisms.

For the morphisms on the left, stronger conditions are necessary (involving the notion of b-analytic space from \cite{Sch23}) to guarantee that they are isomorphisms.

We refer to \cite{HS23} for the full details and more general conditions (some isomorphisms also hold for weakly constructible sheaves).

For the compatibility of hom with tensor products, we have the following natural morphism in general if $\cF_1,\cF_2,\cF_3\in\Dbk{X}$:
\begin{equation}\label{eq:homtensMorph}
\RHom(\cF_1,\cF_2)\otimes \cF_3 \To \RHom(\cF_1,\cF_2\otimes \cF_3).
\end{equation}
This morphism is in particular an isomorphism if $\cF_3$ is locally constant and $\cF_1,\cF_2$ are $\R$-constructible, as shown in \cite[Theorem 4.4(b)]{HS23}.

Let us remark that the results from \cite{HS23} are even more general than what we will use here. In particular, some isomorphisms hold already if some of the sheaves are only \emph{weakly} $\R$-constructible.

\subsection{Direct and inverse images}

Let us first show that we have natural morphisms that are linear over $L$. For this purpose, we give the following two lemmas.

\begin{lemma}\label{lemma:forCommut}
	The functor $\mathsf{for}_{L/K}$ commutes with the functors $\RR f_*$, $\RR f_!$, $f^{-1}$ and $f^!$.
\end{lemma}
\begin{proof}
	We consider a continuous map $f\colon X\to Y$. By the definition of the direct image, it is easy to see that $\mathsf{for}_{L/K}\circ f_*\iso f_*\circ \mathsf{for}_{L/K}$. The statement follows by deriving functors, noting that $\mathsf{for}_{L/K}$ is exact and preserves injectives since it admits an exact left adjoint. We can argue similarly for $\RR f_!$ and $f^{-1}$. For $f^!$, we can proceed by adjunction (using Lemma~\ref{lemma:ScalarExtAdj}):
	\begin{align*}
		\sHom_{\DbK{X}}(\cF,\mathsf{for}_{L/K}f^!\cG)&\iso \sHom_{\DbL{X}}(L_X\otimes_{K_X}\cF,f^!\cG)\\
		&\iso \sHom_{\DbL{Y}}(\RR f_!(L_X\otimes_{K_X}\cF),\cG)\\
		&\iso \sHom_{\DbL{Y}}(L_Y\otimes_{K_Y}\RR f_!\cF,\cG)\\
		&\iso \sHom_{\DbK{Y}}(\RR f_!\cF,\mathsf{for}_{L/K}\cG)\\
		&\iso \sHom_{\DbK{X}}(\cF,f^!\mathsf{for}_{L/K}\cG)
	\end{align*}
	In the third isomorphism, we have already used commutation between scalar extension	and $\RR f_!$, which will be proved (independently of the claim for $f^!$ here) in Proposition~\ref{prop:extIm} below.
\end{proof}
A posteriori, the preceding lemma is the reason why we do not distinguish between the direct and inverse image operations for sheaves over $K$ and $L$ in our notation.

\begin{lemma}\label{lemma:natMorphIm}
	Let $f\colon X\to Y$ be a continuous map. For $\cF\in\DbK{X}$, there exist natural morphisms in $\DbL{Y}$
	$$L_Y\otimes_{K_Y} \RR f_* \cF \to \RR f_* (L_X\otimes_{K_X}\cF), \qquad L_Y\otimes_{K_Y} \RR f_! \cF \to \RR f_! (L_X\otimes_{K_X}\cF),$$
	and for $\cG\in\DbK{Y}$, there exist natural morphisms in $\DbL{X}$
	$$L_X\otimes_{K_X} f^{-1} \cG \to f^{-1} (L_Y\otimes_{K_Y}\cG), \qquad L_X\otimes_{K_X} f^! \cG \to f^! (L_Y\otimes_{K_Y}\cG).$$
\end{lemma}
\begin{proof}
	All the proofs are similar, so let us only prove the first statement. We have
	\begin{align*}
		\sHom_{\DbL{Y}}&\big(L_Y\otimes_{K_Y} \RR f_* \cF, \RR f_* (L_X\otimes_{K_X}\cF)\big)\\
		&\iso \sHom_{\DbK{Y}}\big(\RR f_* \cF, \mathsf{for}_{L/K} \RR f_* (L_X\otimes_{K_X}\cF)\big)\\
		&\iso \sHom_{\DbK{Y}}\big(\RR f_* \cF, \RR f_* \mathsf{for}_{L/K} (L_X\otimes_{K_X}\cF)\big)\\
		&\leftarrow \sHom_{\DbK{X}}\big(\cF, \mathsf{for}_{L/K} (L_X\otimes_{K_X}\cF)\big)\\
		&\iso \sHom_{\DbL{X}}(L_X\otimes_{K_X}\cF,L_X\otimes_{K_X}\cF),
	\end{align*}
	using Lemma~\ref{lemma:ScalarExtAdj} and Lemma~\ref{lemma:forCommut}. The identity element in the last space induces the desired morphism.
\end{proof}

Let us remark that, from what we have recalled in equations \eqref{eq:standardCompat} and \eqref{eq:standardMorph}, morphisms as in the previous lemma certainly exist in $\DbK{Y}$ and $\DbK{X}$. The subtlety of this lemma was therefore to insist on the fact they exist over $L$.

We can now state the compatibilities of extension of scalars with direct and inverse image operations. The first two isomorphisms were already mentioned in \cite{BHHS22}.

\begin{prop}\label{prop:extIm}
	Let $L/K$ be a field extension. Let $f\colon X\to Y$ be a continuous map between topological spaces. Let $\cF\in\DbK{X}$ and $\cG\in\DbK{Y}$. 
	We have isomorphisms over $L$
	\begin{align*}
		L_Y\otimes_{K_Y} \RR f_! \cF &\iso \RR f_! (L_X\otimes_{K_X}\cF),\\
		L_X\otimes_{K_X} f^{-1} \cG &\iso f^{-1} (L_Y\otimes_{K_Y}\cG).
	\end{align*}	
	Consider now the following conditions:
	\begin{itemize}
		\item[(a)] $L/K$ is finite,
		\item[(b)] $X$ and $Y$ are real analytic manifolds, $f\colon X\to Y$ is a morphism of real analytic manifolds, and $\cG\in\DbRcK{Y}$.
	\end{itemize}
	If (a) or (b) is satisfied, then there is an isomorphism in $\DbL{X}$
	\begin{align*}
		L_X\otimes_{K_X} f^! \cG &\iso f^! (L_Y\otimes_{K_Y}\cG).
	\end{align*}
	If (a) is satisfied, we have moreover an isomorphism in $\DbL{Y}$
	$$L_Y\otimes_{K_Y} \RR f_* \cF \iso \RR f_* (L_X\otimes_{K_X}\cF).$$
	All these isomorphisms are induced by the natural morphisms from Lemma~\ref{lemma:natMorphIm}.
\end{prop}
\begin{proof}
	We already know that there are canonical morphisms , linear over $L$, so it remains to check that they are isomorphisms (which can be checked over $K$).
	
	The first two isomorphisms are due to the classical projection formula and the compatibility of inverse images with tensor products (see \eqref{eq:standardCompat}), respectively (note that $f^{-1}L_Y\iso L_X$).
	
	If $L/K$ is finite of degree $n$, we can show the other two isomorphisms as follows: As $K$-vector spaces, we have $L\iso K^n$, so $L_X\otimes_{K_X} (-)\iso \RHom_{K_X}(L_X,-)$ (finite sums are finite products).
	
	Therefore, the isomorphisms follow from the first and last isomorphisms in \eqref{eq:standardCompat}.
	
	If (b) is satisfied, we still get the isomorphism for the exceptional inverse image from \cite[Theorem 4.2(b)]{HS23} (i.e.\ the fourth morphism in \eqref{eq:standardMorph} is an isomorphism in this case).
\end{proof}

\begin{rem}\label{rem:directImage}
	Let us remark that for the direct image, $\R$-constructiblity of $\cF$ is not sufficient to get a commutation with extension of scalars in general.
	
	Let us give a counterexample for the last isomorphism in Proposition~\ref{prop:extIm} in the case of an infinite field extension: Consider $X=\R_{> 0}=\;]0, \infty[\;\subset \R$, $Y=\R$ and $f\colon X\to Y$ the inclusion. Moreover, consider the field extension $\Q\subset \C$. Define the sheaf
	$$\cG \vcentcolon= \prod_{n\in \Z_{>0}} \Q_{\{\frac{1}{n}\}}.$$
	It has stalk $\Q$ at any point $\frac{1}{n}$ and stalk $0$ otherwise. Therefore, we have
	$$\C_X\otimes_{\Q_X}\cG\iso \prod_{n\in \Z_{>0}} \C_{\{\frac{1}{n}\}}.$$
	On the other hand, if $U$ is a small neighbourhood of $0\in Y$ (automatically containing infinitely many points $\frac{1}{n}$), we have
	$$(f_* \cG)(U) \iso \prod_{n\geq N} \Q\iso t^N\Q[\![t]\!]$$
	for some $N$, and hence the stalk is
	$$(f_* \cG)_0 \iso \varinjlim_{N\to \infty} t^N\Q[\![t]\!].$$
	
	Similarly,
	$$\big(f_*(\C_X\otimes_{\Q_X}\cG)\big)_0 \iso \varinjlim_{N\to \infty} t^N\C[\![t]\!],$$
	and this is not isomorphic to $\C\otimes_\Q\varinjlim_{N\to \infty} t^N\Q[\![t]\!]\iso \varinjlim_{N\to \infty} (\C\otimes_\Q t^N\Q[\![t]\!])$: An element of the first vector space is represented by a formal power series that might have infinitely many linearly independent (over $\Q$) coefficients, while an element of the second object needs to be represented by a finite sum of $\Q$-power series multiplied by a complex number, and two such series are equivalent only if they differ in a finite number of coefficients.
	
	In this example, $\cF$ is indeed $\R$-constructible, but its extension to a compactification of $X$ is not. For a statement about direct images in the case of infinite extensions, we need to require that $X$, $Y$, $f$ and $\cF$ fit into the context of b-constructible sheaves (see \cite{Sch23} and \cite[Theorem~4.7]{HS23}).
\end{rem}

\subsection{Homomorphisms between scalar extensions}

We will now study the functors $\RHom$ and $\sHom$. As for the direct and inverse images above, let us first establish natural morphisms which are $L$-linear.

\begin{lemma}\label{lemma:natMorph}
	Let $L/K$ be a field extension, $X$ a topological space and $\cF,\cG\in\DbK{X}$. There is a natural morphism in $\DbL{X}$
	$$L_X\otimes_{K_X} \RHom_{K_X}(\cF,\cG) \to \RHom_{L_X}(L_X\otimes_{K_X}\cF,L_X\otimes_{K_X}\cG)$$
	and a natural morphism of $L$-vector spaces
	$$L\otimes_{K} \sHom_{\DbK{X}}(\cF,\cG) \to \sHom_{\DbL{X}}(L_X\otimes_{K_X}\cF,L_X\otimes_{K_X}\cG).$$
\end{lemma}
\begin{proof}
	Using Lemma~\ref{lemma:ScalarExtAdj} repeatedly, we have isomorphisms and a morphism
	\begin{align*}
		\sHom&_{\DbL{X}}\big(L_X\otimes_{K_X} \RHom_{K_X}(\cF,\cG),\RHom_{L_X}(L_X\otimes_{K_X}\cF,L_X\otimes_{K_X}\cG)\big)\\
		&\iso \sHom_{\DbK{X}}\big(\RHom_{K_X}(\cF,\cG),\mathsf{for}_{L/K}\RHom_{L_X}(L_X\otimes_{K_X}\cF,L_X\otimes_{K_X}\cG)\big)\\
		&\iso \sHom_{\DbK{X}}\big(\RHom_{K_X}(\cF,\cG),\RHom_{K_X}(\cF,\mathsf{for}_{L/K}(L_X\otimes_{K_X}\cG))\big)\\
		&\leftarrow \sHom_{\DbK{X}}\big(\cG,\mathsf{for}_{L/K}(L_X\otimes_{K_X}\cG)\big)\\
		&\iso \sHom_{\DbL{X}}(L_X\otimes_{K_X}\cG,L_X\otimes_{K_X}\cG)
	\end{align*}
	and the desired morphism is induced by the identity element of the last space.
	
	The second morphism follows then by taking zeroth cohomology (note that extension of scalars is exact) and global sections. To conclude, note that $L_X\otimes_{K_X} \Hom(\cF,\cG)$ is the sheafification of a presheaf with global sections $L\otimes_K \sHom(\cF,\cG)$, and there is a natural morphism from a presheaf to its associated sheaf.
\end{proof}

We now describe when this morphism is an isomorphism. The statement of the following proposition was already given in \cite{BHHS22}, where it was derived from a result about ind-sheaves. In a previous version of this article, we gave a similar, but more direct argument for sheaves, using a description of $\R$-constructible sheaves via simplicial complexes. Here, we will avoid these arguments by using the results of \cite{HS23}.

\begin{prop}[{see \cite[Lemma 2.7]{BHHS22}}]\label{prop:BHHSsheafhom}
	Let $L/K$ be a field extension. Let $X$ be a topological space, and let $\cF,\cG\in\DbK{X}$.
	Consider the following conditions:
	\begin{itemize}
		\item[(a)] $L/K$ is finite.
		\item[(b)] $X$ is a real analytic manifold and $\cF,\cG\in\DbRcK{X}$.
	\end{itemize}
	If (a) or (b) are satisfied, then there is an isomorphism in $\DbL{X}$
	$$\RHom_{L_X}(L_X\otimes_{K_X}\cF,L_X\otimes_{K_X}\cG) \iso L_X\otimes_{K_X} \RHom_{K_X}(\cF,\cG).$$
\end{prop}
\begin{proof}
	There is a natural morphism in $\DbL{X}$ between the two objects by Lemma~\ref{lemma:natMorph}. It suffices to prove that it is an isomorphism over $K$ (i.e.\ after applying $\mathsf{for}_{L/K}$), since the notion of isomorphism of sheaves does not depend on the field (one could check it on the level of sheaves of sets). By Lemma~\ref{lemma:ScalarExtAdj}, it is therefore enough to prove the isomorphism in $\DbK{X}$
	$$\RHom_{K_X}(\cF,L_X\otimes_{K_X}\cG) \iso L_X\otimes_{K_X} \RHom_{K_X}(\cF,\cG).$$
	If $L/K$ is finite, this follows from $L_X\otimes_{K_X}(-)\iso \RHom_{K_X}(L_X,-)$ and the well-known isomorphism
	$$\RHom_{K_X}(\cF,\RHom_{K_X}(L_X,\cG))\iso \RHom_{K_X}(L_X,\RHom_{K_X}(\cF,\cG))$$
	(see \cite[Proposition 2.6.3]{KS90}).
	
	If (b) is satisfied, this follows from \cite[Theorem 4.4(b)]{HS23} (i.e.\ the morphism \eqref{eq:homtensMorph} is an isomorphism if (b) is satisfied).
\end{proof}

With this in hand, we can now state the description of the space of homomorphisms between extensions of scalars.
\begin{prop}\label{prop:morphLification}
	Let $L/K$ be a field extension. Let $X$ be a topological space, and let $\cF,\cG\in\DbK{X}$.
	Consider the following conditions:
	\begin{itemize}
		\item[(a)] $L/K$ is finite.
		\item[(b)] $X$ is a compact real analytic manifold and $\cF,\cG\in\DbRcK{X}$.
	\end{itemize}
	If (a) or (b) is satisfied, then we have an isomorphism of $L$-vector spaces 
	$$\sHom_{\DbL{X}}(L_X\otimes_{K_X}\cF, L_X\otimes_{K_X}\cG) \iso L\otimes_{K}\sHom_{\DbK{X}}(\cF, \cG).$$
\end{prop}
\begin{proof}
	A natural $L$-linear morphism exists by Lemma~\ref{lemma:natMorph}. The isomorphism is obtained from Proposition~\ref{prop:morphLification} by taking (derived) global sections and zeroth cohomology. The functor of derived global sections is a direct image functor $\RR {p_X}_*$ along the map to a one-point space, and we need to see that $\RR {p_X}_*(L_X\otimes_{K_X}(-))\iso L\otimes_K \RR {p_X}_*(-)$. If $L/K$ is finite, this holds by \eqref{eq:standardCompat} since $L_X\otimes_{K_X}(-)\iso \RHom_{K_X}(L_X,-)$. If $X$ is compact, we have $\RR {p_X}_* \iso \RR {p_X}_!$ and hence it follows from Proposition~\ref{prop:extIm} (noting that $\RHom_{K_X}(\cF,\cG)$ is $\R$-constructible if $\cF$ and $\cG$ are, see \cite[Proposition~8.4.10]{KS90}).
\end{proof}

\begin{rem}
	In view of \cite[Theorem~4.7]{HS23}, we can generalize condition (b) in Proposition~\ref{prop:morphLification} as follows, using the notion of b-constructible sheaves from \cite{Sch23}:
	\begin{itemize}
		\item[(b')] $(X,\widehat{X})$ is a b-analytic manifold and $\cF,\cG\in\DbRcK{X_\infty}$.
	\end{itemize}
	Then, by Proposition~\ref{prop:extIm}, the assertion of the above proposition holds if (b') is satisfied, and (b) is the special case of (b') where $X$ is a compact real analytic manifold and $\widehat{X}=X$.
\end{rem}

Compared to general field extensions, the case of a (finite) Galois extension has an additional feature, the Galois group, which in a certain sense ``knows enough'' about the field extension. Using it, we can describe more intrinsically the image of morphisms under extension of scalars.

\begin{prop}\label{prop:morphLificationFin}
	Let $L/K$ be a finite Galois extension. Let $X$ be a topological space, and let $\cF,\cG\in\DbK{X}$.
	Then the subspace of
	\begin{equation*}
		\sHom_{\DbL{X}}(L_X\otimes_{K_X}\cF, L_X\otimes_{K_X}\cG) \iso L\otimes_{K}\sHom_{\DbK{X}}(\cF, \cG).
	\end{equation*}
	consisting of morphisms $f$ which fit into the natural diagram
	$$\begin{tikzcd}
		L_X\otimes_{K_X}\cF \arrow{r}{f} \arrow{d}{g\otimes\id_{\cF}}[swap]{\iso} & L_X\otimes_{K_X}\cG \arrow{d}{g\otimes\id_\cG}[swap]{\iso}\\
		\overline{L_X}^g\otimes_{K_X}\cF \arrow{r}{\overline{f}^g} &\overline{L_X}^g\otimes_{K_X}\cG
	\end{tikzcd}$$
	for any $g\in G$ is exactly the subset of morphisms of the form $1\otimes f_K$ for some $f_K\in\sHom_{\DbK{X}}(\cF,\cG)$.
\end{prop}

\begin{proof}
	Let $\cF,\cG\in\DbK{X}$. Choose a $K$-basis $\ell_1,\ldots,\ell_d$ of $L$. By Proposition~\ref{prop:morphLification}, an element of $\sHom_{\DbL{X}}(\cF\otimes_{K_X}L_X, \cG\otimes_{K_X}L_X)$ can uniquely be written in the form $f=\sum_{k=1}^{d} \ell_k\otimes f_k$ for some $f_k\in \sHom_{\DbK{X}}(\cF,\cG)$ (where $\ell_k$ is understood as the map $L_X\to L_X$ given by multiplication with $\ell_k$).
	Assume that $f$ fits in a diagram as above for any $g\in G$. The vertical isomorphisms are induced by the natural $G$-structure on the constant sheaf $L_X$. Commutation of this diagram therefore means that $\sum_{k=1}^{d} g(\ell_k\cdot (-))\otimes f_k = \sum_{k=1}^{d} (g(\ell_k)\cdot g(-)) \otimes f_k$ (going right and then down in the diagram) and $\sum_{k=1}^{d} (\ell_k\cdot g(-)) \otimes f_k$ (going down and then right) coincide as morphisms in $\sHom_{\DbL{X}}(L_X\otimes_{K_X}\cF,\overline{L_X}^g\otimes_{K_X}\cG)$. Now there is an isomorphism 
	\begin{align*}
		\sHom_{\DbL{X}}(L_X\otimes_{K_X}\cF,\overline{L_X}^g\otimes_{K_X}\cG)&\iso \sHom_{\DbL{X}}(L_X\otimes_{K_X}\cF,L_X\otimes_{K_X}\cG)\\
		&\iso L\otimes_K \sHom_{\DbK{X}}(\cF,\cG)
	\end{align*} 
	with the first isomorphism given by $\varphi\mapsto (g^{-1}\otimes \id)\circ\varphi$, hence the above means that for any $g\in G$ we have $\sum_{k=1}^{d} \ell_k \otimes f_k= \sum_{k=1}^{d} g(\ell_k) \otimes f_k$ in $L\otimes_K \sHom_{\DbK{X}}(\cF,\cG)$. Then the result follows from the fact that the invariants of the Galois action on $L\otimes_K V$ (for a $K$-vector space $V$) are given by the $K$-subspace $1\otimes V$ (cf. e.g. \cite[Corollary 2.17]{Conrad}).
\end{proof}

Altogether, we have in particular proved the following property.
\begin{prop}
	Let $L/K$ be a field extension and $X$ a topological space.
	\begin{itemize}
		\item[(a)] If $L/K$ is finite, the functor $\Phi_{L/K}=L_X\otimes_{K_X}(-)$ is faithful on $\DbK{X}$.
		\item[(b)] If $X$ is a compact real analytic manifold, then $\Phi_{L/K}$ is faithful on $\DbRcK{X}$.
	\end{itemize}
\end{prop}
\begin{proof}
	These statements follow directly from Proposition~\ref{prop:morphLification} (a) and (b), respectively.
\end{proof}

Finally, we can use our findings of this subsection to establish a compatibility of extension of scalars with the duality functor. 
\begin{corr}[of Proposition~\ref{prop:BHHSsheafhom}]\label{cor:dualityExt}
	Let $L/K$ be a field extension. Let $X$ be a topological space, and let $\cF\in\DbK{X}$.
	If $L/K$ is finite, or if $X$ is a real analytic manifold and $\cF\in\DbRcK{X}$, then there is an isomorphism in $\DbL{X}$ 
		$$\DD_X^L(L_X\otimes_{K_X}\cF)\iso L_X\otimes_{K_X} \DD^K_X\cF.$$
\end{corr}
\begin{proof}
	It follows from \cite[Proposition 3.3.4]{KS90} that we have the relation $\omega_X^L\iso L_X\otimes_{K_X}\omega_X^K$ between the dualizing complexes in $\DbK{X}$ and $\DbL{X}$. Then we get
		\begin{align*}
			\DD_X^L(L_X\otimes_{K_X}\cF) &= \RHom_{L_X}(L_X\otimes_{K_X}\cF,\omega_X^L)\\
			&\iso \RHom_{L_X}(L_X\otimes_{K_X}\cF,L_X\otimes_{K_X} \omega^K_X)\\
			&\iso L_X\otimes_{K_X} \RHom_{K_X}(\cF,\omega^K_X)=L_X\otimes_{K_X} \DD_X^K\cF.
		\end{align*}
		Here, we have applied Proposition~\ref{prop:BHHSsheafhom} in the third line.
\end{proof}

\subsection{$K$-structures and constructibility}

Let $L/K$ be a field extension, and let $X$ be a topological space.

It is not difficult to see that extension of scalars preserves (local) constancy, constructibility and perversity, and that these properties descend to $K$-lattices. We will state these properties here.
\begin{lemma}\label{lemma:latticeLocSys}
	Let $X$ be a topological space. If $\cG\in\Mod{K_X}$, then $\cG$ is locally constant of finite rank if and only if $\cF\vcentcolon=L_X\otimes_{K_X}\cG\in\Mod{L_X}$ is locally constant of finite rank.
\end{lemma}
\begin{proof}
		If $\cG$ is locally constant, it is easily deduced from Proposition~\ref{prop:extIm} that $\cF$ is so, since constant sheaves are inverse images of vector spaces along the map to the one-point space. (Let us remark that for this direction it is not even necessary that $\cG$ is of finite rank.)
		
		For the other direction, we reproduce the proof in \cite[Proof of Lemma~2.13]{BHHS22} slightly differently here for completeness.
		
		A local system is characterized by the fact that each $x\in X$ has an open neighbourhood $U\subseteq X$ such that $\cF(U)\to \cF_x$ is surjective and such that for any $y\in U$ there is an isomorphism $\cF_x\overset{\sim}{\to}\cF_y$ such that the diagram
		$$\begin{tikzcd}
			&\cF(U)\arrow[two heads]{ld} \arrow{rd}\\
			\cF_x \arrow{rr}{\sim} & & \cF_y
		\end{tikzcd}$$
		commutes. (This is a nice exercise in basic sheaf theory.) We are given this property for $\cF$ and want to prove the analogous property for $\cG$.
		This follows from the commutative diagram
		$$\begin{tikzcd}
			&&L\otimes_K \cG(U)\arrow[bend right]{llddd} \arrow[bend left]{rrddd}\arrow{d}\\
			&&\cF(U)\arrow{rd} \arrow[two heads]{ld}\\
			&\cF_x \arrow{rr}{\sim} & & \cF_y\\
			L\otimes_K \cG_x \arrow{ru}{\sim} \arrow[dashed]{rrrr}{} & & & & L\otimes_K \cG_y.\arrow{ul}[swap]{\sim}
		\end{tikzcd}$$
		To see this, observe the following: Since $\cF_x$ is a finite-dimensional vector space (over $L$), so is $\cG_x$ (over $K$), and hence (by possibly shrinking $U$, but this does not affect the surjectivity of $\cF(U)\to \cF_x$) the restriction map $\cG(U)\to \cG_x$ is surjective because each of the (finitely many) generators of $\cG_x$ is induced by a section on a neighbourhood of $x$. Therefore, the dashed morphism descends to a morphism $\cG_x\to\cG_y$: It is induced by the map taking an element of $\cG_x$ and mapping a preimage in $\cG(U)$ of it to $\cG_y$ via the restriction map $\cG(U)\to\cG_y$. Then, we remark that this map $\cG_x\to \cG_y$ is an isomorphism since the dashed arrow is so and tensor products over fields are exact.
\end{proof}

\begin{rem}
	Let us briefly reflect on the case of local systems with possibly infinite rank. Let $\cG\in \Mod{K_X}$.
	
	As already remarked in the proof of the above lemma, the local constancy of $L_X\otimes_{K_X} \cG$ follows from that of $\cG$ even without the hypothesis that the rank of $\cG$ is finite.
	
	For the other direction, we can prove the following result: If $L/K$ is a finite field extension and $\cF\vcentcolon= L_X\otimes_{K_X}\cG$ is locally constant, then $\cG$ is locally constant.
	
	The proof is mostly analogous to the one above. The only argument that needs to be changed (because it explicitly used the finite rank property) is the one for the surjectivity of the morphism $\cG(U)\to \cG_x$. But in the case where $L/K$ is finite, we have an isomorphism $L\otimes_K \cG(U)\iso \cF(U)$, which follows from Proposition~\ref{prop:extIm} since the functor $\Gamma(U;-)$ taking sections on $U$ is nothing but $H^0 \mathrm{R}{p_U}_* j_U^{-1}$, where $j_U\colon U\hookrightarrow X$ is the inclusion and $p_U\colon U\to \{\mathrm{pt}\}$ is the map to the one-point space. Consequently, the surjectivity of $\cG(U)\to\cG_x$ follows from that of $\cF(U)\to\cF_x$ and the exactness of tensor products over fields. The rest of the proof is analogous.
\end{rem}

\begin{corr}
	Let $X$ be a real analytic manifold, let $\cG\in\DbK{X}$ and set $\cF\vcentcolon= L_X\otimes_{K_X} \cG\in\DbL{X}$.
	\begin{itemize}
		\item[(a)] We have $\cF\in\DbRcL{X}$ if and only if $\cG\in \DbRcK{X}$. Accordingly, if $X$ is a complex manifold, then we have $\cF\in\DbCcL{X}$ if and only if $\cG\in\DbCcK{X}$.
		\item[(b)] If $X$ is a complex manifold, then $\cF\in\PervL{X}$ if and only if $\cG\in\PervK{X}$.
	\end{itemize}
\end{corr}
\begin{proof}
	\begin{itemize}
		\item[(a)] can be derived directly from Lemma~\ref{lemma:latticeLocSys}, since restriction to strata and cohomology commute with extension of scalars.
		
		Let us, however, give a more intrinsic argument without choosing a stratification: By \cite[Proposition~4.12]{HS23}, the microsupports of $\cF$ and $\cG$ are the same. We know that being $\R$- or $\C$-constructible depends only on the microsupport, and hence this condition is satisfied for $\cF$ if and only if it is for $\cG$.
		\item[(b)] Since extension of scalars is exact, commutes with taking stalks and is faithful, the supports of the cohomologies of $\cF$ are the same as those for $\cG$ (and the same holds for $\DD_X\cF$ and $\DD_X\cG$ in view of Corollary~\ref{cor:dualityExt}). Hence the support condition for $\cG$ and $\DD_X\cG$ holds if and only if it holds for $\cF$ and $\DD_X\cF$.
	\end{itemize}
\end{proof}

\section{Galois descent for sheaves and their complexes}\label{sec:GaloisDescent}

Let $L/K$ be a finite Galois extension with Galois group $G$. In this section, we will first formulate Galois descent for sheaves of vector spaces. Afterwards, we are going to investigate a similar procedure for derived categories of sheaves of vector spaces, and we will describe what kind of problems arise there and prevent us from obtaining an equally nice result. Finally, we restrict ourselves to a particularly nice subcategory of the derived category of sheaves of vector spaces, namely that of perverse sheaves, and we establish Galois descent for them, using a construction performed by A.\ Beilinson in \cite{Bei}.

\subsection{Galois descent for sheaves of vector spaces}

Galois descent for sheaves of vector spaces has already been studied in \cite{BHHS22}, but the statement was not formulated as an equivalence of categories in loc.~cit. We reformulate it here to fit into the framework set up in Section~\ref{subsec:Galois}, using our results from Section~\ref{sec:Hom}.

\begin{prop}[{cf.\ \cite[Lemma 2.13]{BHHS22}}]\label{prop:GaloisDescentSheaves}
	Let $L/K$ be a finite Galois extension with Galois group $G$, and let $X$ be a topological space.
	
	For any $\cF\in\Mod{L_X}$ equipped with a $G$-structure $(\varphi_g)_{g\in G}$, there exists a $K$-structure $\cG\in\Mod{K_X}$ together with a natural isomorphism $\psi\colon L_X\otimes_{K_X} \cG\overset{\sim}{\to} \cF$ such that the natural $G$-structure on $L_X\otimes_{K_X}\cG$ corresponds via this isomorphism to the given one on $\cF$, i.e.\ such that, for any $g\in G$, the following diagram commutes:
	$$\begin{tikzcd}
		\cF\arrow{rr}{\varphi_g} && \overline{\cF}^g\\ \\
		L_X\otimes_{K_X}\cG\arrow{uu}{\psi}\arrow{rr}{g\otimes\id_\cG} && \overline{L_X}^g\otimes_{K_X}\cG\arrow{uu}{\overline{\psi}^g}
	\end{tikzcd}$$
	In particular, the functor $\Phi_{L/K}=L_X\otimes_{K_X}(-)$ induces an equivalence
	$$\Phi_{L/K}^G\colon \Mod{K_X}\To \Mod{L_X}^G.$$
\end{prop}
\begin{proof}
	We proceed analogously to the construction for vector spaces. We consider $\cF$ as a sheaf of $K$-vector spaces and each $\varphi_g$ as a $K$-linear automorphism of $\cF$. Then $\cG$ is defined as
	$$\cG\vcentcolon= \ker \big(\prod\limits_{g\in G} (\varphi_g^K-\id_{\cF^K})\colon \cF^K\To \prod_{g\in G} \cF^K\big)\in\Mod{K_X}$$
	together with a canonical $L$-linear morphism $L_X\otimes_{K_X} \cG\to \cF$ due to Lemma~\ref{lemma:ScalarExtAdj}.
	Since kernels and extension of scalars commute with taking stalks, the isomorphism $L_X\otimes_{K_X}\cG\iso \cF$ is obtained from Galois descent for vector spaces (cf.\ e.g.\ Section~\ref{subsec:Galois} and \cite{Conrad}). The commutation of the above square is also clear from the construction of $\cG$ as a sheaf of invariants.
	
	This proves in particular essential surjectivity of the functor $\Phi_{L/K}^G$.
	Full faithfulness follows from Propositions~\ref{prop:morphLification} and \ref{prop:morphLificationFin}.
\end{proof}

\begin{rem}
	We can also similarly establish a Galois descent statement for functor categories as in Example~\ref{ex:GConj} (b). They admit an obvious functor of extension of scalars $\mathrm{Funct}(\cC,\Vect{K})\to \mathrm{Funct}(\cC,\Vect{L})$. This yields then in particular Galois descent for presheaves of vector spaces (if we take $\cC=\mathrm{Op}(X)^\op$), and to deduce Proposition~\ref{prop:GaloisDescentSheaves}, one just needs to check that the descent of a sheaf is still a sheaf. (The extension of scalars for sheaves is indeed the same as the one for presheaves in the finite Galois case due to Proposition~\ref{prop:extIm}.)
	
	Choosing more complicated categories $\cC$ (so-called categories of exit paths), one can also express certain categories of constructible sheaves as functor categories. This technique is known as \emph{exodromy} (see e.g.\ \cite{Treumann} and \cite{BGH}), and it has also been set up in the framework of quasi-categories (see \cite{Lurie}, and \cite{PT22} for a recent generalization). Such exodromy equivalences might serve as an alternative approach to our questions for certain constructible sheaves, and they might also lead to a clearer study of complexes of sheaves. We will, however, not take this viewpoint here. We are grateful to Jean-Baptiste Teyssier for drawing our attention to these constructions.
\end{rem}

\subsection{The case of derived categories of sheaves}

Having established the above equivalence (Theorem~\ref{prop:GaloisDescentSheaves}) for sheaves, one would, of course, like to generalize such a statement to derived categories of sheaves. Indeed, for a finite Galois extension $L/K$ with Galois group $G$, we still have the functor induced by extension of scalars
\begin{align*}
	\Phi_{L/K}^G\colon \DbK{X} &\To \DbL{X}^G.
\end{align*}
The following is deduced directly -- as above -- from Propositions~\ref{prop:morphLification} and \ref{prop:morphLificationFin}.
\begin{prop}\label{prop:ExtDerivedFF}
	Let $L/K$ be a finite Galois extension. The functor $\Phi_{L/K}^G$ is fully faithful on $\DbK{X}$.
\end{prop}

\begin{rem}
	Essential surjectivity does not seem to hold in this more general situation. An indication for why this makes sense is the following: To an object $\cF^\bullet=\ldots\to\cF_i\to \cF_{i+1}\to\ldots$, the functor associates the object $L_X\otimes_{K_X}\cF^\bullet = \ldots\to L_X\otimes_{K_X}\cF_i\to L_X\otimes_{K_X}\cF_{i+1}\to\ldots$ (due to the exactness of the tensor product), together with the natural $G$-structure given by that on $L_X$, i.e.\
	$$\begin{tikzcd}
		\ldots \arrow{r} & L_X\otimes_{K_X}\cF_i \arrow{r}\arrow{d}{g\otimes\id} & L_X\otimes_{K_X}\cF_{i+1} \arrow{r}\arrow{d}{g\otimes\id} & \ldots\\
		\ldots \arrow{r} & \overline{L_X}^g\otimes_{K_X}\cF_i \arrow{r} & \overline{L_X}^g\otimes_{K_X}\cF_{i+1} \arrow{r} & \ldots
	\end{tikzcd}$$
	In particular, this $G$-structure is given by morphisms of complexes (rather than roofs, as is the general case for morphisms in the derived category). In general, morphisms $\varphi_g$ in the derived category can -- after choosing suitable resolutions -- be represented as morphisms of complexes, but the compatibilities that the $\varphi_g$ have to satisfy will only hold up to homotopy of morphisms of complexes, and hence in general such a $G$-structure will not be in the essential image of $\Phi_{L/K}^G$.
\end{rem}

A standard technique for proofs of statements in derived categories is by induction on the amplitude of a complex. At least for complexes concentrated in two successive degrees, we can use this approach to deduce some existence statement of a $K$-lattice from the existence of a $G$-structure.

\begin{prop}\label{prop:descentDerived}
	Let $L/K$ be a finite Galois extension with Galois group $G$, and let $X$ be a topological space. Let $\cF^\bullet\in \DbL{X}$ be equipped with a $G$-structure $(\varphi_g)_{g\in G}$. Assume that $\cF^\bullet$ is concentrated in degrees $a$ and $a+1$ (i.e.\ the only non-vanishing cohomologies are $H^a(\cF^\bullet)$ and $H^{a+1}(\cF^\bullet)$). Then there exist $\cF^\bullet_K\in \DbK{X}$ and an isomorphism $L_X\otimes_{K_X} \cF^\bullet_K \iso \cF^\bullet$.
\end{prop}
\begin{proof}
	We know the statement for sheaves (i.e.\ complexes concentrated in one degree) from Proposition~\ref{prop:GaloisDescentSheaves}, so we know that $H^a(\cF^\bullet)\iso L_X\otimes_{K_X}\cG_a$ and $H^{a+1}(\cF^\bullet)\iso L_X\otimes_{K_X}\cG_{a+1}$ such that under these isomorphisms the $G$-structures on the $H^i(\cF^\bullet)$ induced by the one on $\cF^\bullet$ coincide with those given by the natural $G$-structure on $L_X$.
	
	Using the standard truncation functors for complexes (with respect to the standard t-structure on $\DbL{X}$), there is a distinguished triangle
	$$H^{a+1}(\cF^\bullet)[-a-2]\To H^a(\cF^\bullet)[-a] \To \cF^\bullet \ToPO$$
	and for any $g\in G$ the $G$-structure on $\cF^\bullet$ induces an isomorphism of distinguished triangles
	\begin{equation}\label{eq:diagTruncation}\begin{tikzcd}
			H^{a+1}(\cF^\bullet)[-a-2]\arrow{r}{f}\arrow{d}[swap]{\iso} & H^a(\cF^\bullet)[-a] \arrow{r} \arrow{d}[swap]{\iso} & \cF^\bullet \arrow{r}{+1} \arrow{d}{\varphi_g}[swap]{\iso} & \text{ }\\
			\overline{H^{a+1}(\cF^\bullet)[-a-2]}^g\arrow{r} & \overline{H^a(\cF^\bullet)[-a]}^g \arrow{r} & \overline{\cF^\bullet}^g \arrow{r}{+1} & \text{ }
	\end{tikzcd}\end{equation}
	(note that truncation and conjugation commute).
	
	By Proposition~\ref{prop:morphLificationFin}, there is a morphism $\tilde{f} \colon \cG_{a+1}[-a-2]\to \cG_{a}[-a]$ such that $f=1\otimes \tilde{f}$. We can complete it to a distinguished triangle
	$$\cG_{a+1}[-a-2]\overset{\tilde{f}}{\To} \cG_{a}[-a] \To \cG^\bullet \ToPO$$
	where is $\cG^\bullet$ is unique up to (non-unique) isomorphism. Hence, there exists a (non-unique) isomorphism $\gamma\colon L_X\otimes_{K_X}\cG^\bullet \overset{\sim}{\to} \cF^\bullet$. In other words, we have completed \eqref{eq:diagTruncation} to a commutative diagram
	
	$$\begin{tikzcd}[scale cd=0.7]
		&[-15pt]L_X\otimes_{K_X}\cG_{a+1}[-a-2]\arrow{ld}[swap]{\iso}\arrow{dd}[crossing over]{} \arrow{rr}{g\otimes 1}&[-15pt]&[-15pt] L_X\otimes_{K_X}\cG_{a}[-a]\arrow{ld}[swap]{\iso}\arrow{dd}[crossing over]{} \arrow{rr} &[-15pt]&[-10pt] L_X\otimes_{K_X}\cG^\bullet\arrow{ld}[swap]{\iso}\arrow[dashed,crossing over]{dd} \arrow{r}{+1} &[-10pt] \text{ } \\
		H^{a+1}(\cF^\bullet)[-a-2]\arrow{rr}[near start]{f}\arrow{dd}[swap]{\iso} && H^a(\cF^\bullet)[-a] \arrow{rr}  && \cF^\bullet \arrow{r}{+1} & \text{ }\\
		&\overline{L_X}^g\otimes_{K_X}\cG_{a+1}[-a-2]\arrow{ld}[swap]{\iso} \arrow{rr}{} && \overline{L_X}^g\otimes_{K_X}\cG_{a}[-a]\arrow{ld}[swap]{\iso} \arrow{rr}{} &&\overline{L_X}^g\otimes_{K_X}\cG^\bullet\arrow{ld}[swap]{\iso} \arrow{r}{+1} & \text{ } \\
		\overline{H^{a+1}(\cF^\bullet)[-a-2]}^g\arrow{rr} && \overline{H^a(\cF^\bullet)[-a]}^g \arrow{rr} \arrow[crossing over,<-]{uu}[near end]{\iso} && \overline{\cF^\bullet}^g \arrow{r}{+1} \arrow[crossing over,<-]{uu}[near end]{\iso}[swap,near end]{\varphi_g} & \text{ }
	\end{tikzcd}$$
\end{proof}

\begin{rem}
	The construction in the proof above is highly non-canonical. This is due to the fact that the third objects and morphisms in distinguished triangles are not unique up to unique isomorphism. In particular, although we find an object $\cG^\bullet$ with $L_X\otimes_{K_X}\cG^\bullet\iso \cF^\bullet$ here, it seems not clear that the given $G$-structure corresponds to the natural one on $L_X\otimes_{K_X} \cG^\bullet$. In other words, it is not clear that the dashed arrow in the big diagram is given by $\id_{\cG^\bullet}\otimes g$. This has two implications:
	\begin{itemize}
		\item First, this proposition does not show that the object $(\cF,(\varphi_g)_{g\in G})$ is in the essential image of the functor $L_X\otimes_{K_X} (-)$. It just shows that $\cF^\bullet$ can be realized as such a tensor product if there exists a $G$-structure on it (but not saying that exactly this $G$-structure comes through the tensor product).
		\item Secondly, we cannot proceed inductively to get similar results for complexes concentrated in more than two degrees since the fact that the $G$-structure corresponds to the natural one on the tensor product is crucial for descending the morphism $f$ to $\tilde{f}$.
	\end{itemize}
\end{rem}

\subsection{Galois descent for perverse sheaves}

It is reasonable to expect that the results are closer to the statements for sheaves if we do not consider general objects of the derived category, but perverse sheaves.\footnote{Thinking in terms of Algebraic Analysis, perverse sheaves are the counterparts of regular holonomic D-modules -- objects concentrated in one degree -- via the Riemann--Hilbert correspondence (see \cite{KasRHreg}), so these are the objects to understand if one wants to understand topologically the category of regular holonomic D-modules.}
\begin{rem}
	Galois descent for perverse sheaves seems to be a well-known result. In the case $\C/\R$, an equivalence as in Theorem~\ref{thm:GaloisDescentPerv} is explicitly mentioned in \cite[§12.6.2.5]{Moc15}. Moreover, Theorem~\ref{thm:GaloisDescentPerv} should be a special case of the much more general framework developed in \cite{PT22}.
\end{rem}
From now on, we work in the context of complex analytic varieties, which is a generalization of (but often analogous to) the theory of complex manifolds, but it is a more natural setup for Beilinson's construction that we use later.

Let $L/K$ be a finite Galois extension, and let $X$ be a complex analytic variety.
The notion of $G$-conjugation on $\PervL{X}$ as well as the functors of extension and restriction of scalars on perverse sheaves are inherited from those between $\DbK{X}$ and $\DbL{X}$.

Given a perverse sheaf $\cF\in\PervL{X}$ together with a $G$-structure $(\varphi_g)_{g\in G}$, we can define the presumptive $K$-structure similarly to the case of vector spaces and sheaves, namely as an object of invariants. Contrarily to the case of derived categories, we dispose of the notion of kernels here, since perverse sheaves form an abelian category.

Consider the underlying $K$-perverse sheaf $\cF^K\in\PervK{X}$ (which is the same for any $\overline{\cF}^g$) with the automorphisms $\varphi^K_g\colon \cF^K\to\cF^K$ induced by the $\varphi_g$.
Then define $$\cG\vcentcolon= \ker\big(\prod_{g\in G}(\varphi_g^K-\id_{\cF^K})\colon \cF^K\To \prod_{g\in G} \cF^K\big).$$ 
Clearly, we have a morphism $\cG\to \cF^K$, and hence by Lemma~\ref{lemma:ScalarExtAdj}, we have a morphism
\begin{equation}\label{eq:latticeComplPerv}
	L_X\otimes_{K_X}\cG \to \cF.
\end{equation}

We will give a proof of the fact that it is an isomorphism. For this, we will use a description of perverse sheaves due to Beilinson \cite{Bei} (see also \cite{Rei} for some more details and complements on Beilinson's article).

\paragraph{Beilinson's equivalence} Let us recall the idea of Beilinson's ``gluing'' of perverse sheaves, and study the properties of his equivalence with respect to field extensions. We will not review all the details, for which we refer to \cite{Bei} and \cite{Rei}.

Let $k$ be a field, $X$ a complex analytic variety and $f\colon X\to\C$ a holomorphic function. Write $Z\vcentcolon=f^{-1}(0)$ and $U\vcentcolon= X\setminus Z$ with inclusion $j\colon U\hookrightarrow X$. Moreover, we fix a generator $t$ of the fundamental group $\pi_1(\C\setminus\{0\})$.

There is a functor $\NCunf\colon \Pervk{U}\to \Pervk{Z}$ of \emph{unipotent nearby cycles}. By its construction, the fundamental group $\pi_1(\C\setminus\{0\})$ acts on $\NCunf(\cF_U)$ for any $\cF_U\in\Pervk{U}$. In particular, the fixed generator $t$ induces an endomorphism of any such $\NCunf(\cF_U)$, which we will still denote by $t$.

There is also a functor $\VCunf\colon \Pervk{X}\to\Pervk{Z}$ of \emph{unipotent vanishing cycles}. 

One defines the category of \emph{gluing data} $\GDXfk$ to be the category whose objects are tuples $(\cF_U,\cF_Z,u,v)$, where $\cF_U\in\Pervk{U}$, $\cF_Z\in\Pervk{Z}$ and $u$ and $v$ are morphisms
$$\NCunf(\cF_U)\overset{u}{\to} \cF_Z \overset{v}{\to} \NCunf(\cF_U)$$
such that their composition coincides with the endomorphism $(\id-t)$ of $\NCunf(\cF_U)$.
A morphism in $\GDXfk$ is defined in the obvious way, as two morphisms of perverse sheaves on $U$ and $Z$, respectively, making the natural diagram with the $u$ and $v$ commute.

It is shown (see \cite[Proposition 3.1]{Bei} or \cite[Theorem 3.6]{Rei}) that $\GDXfk$ is an abelian category and that there is an equivalence
$$\Ff\colon \Pervk{X}\overset{\sim}{\To} \GDXfk,$$
sending $\cF\in\Pervk{X}$ to the tuple $(j^{-1}\cF,\VCunf(\cF),u,v)$, where we will not go into detail with the construction of the maps $u$ and $v$. Also the quasi-inverse is explicitly described.

Let us now study this equivalence in the context of a field extension  $L/K$. The construction of the nearby and vanishing cycles functor as well as of the maps $u$ and $v$ are completely topological and do not depend on the coefficient field (they could as well be performed on sheaves of sets). Therefore, if $L/K$ is a finite field extension, the forgetful functor (restriction of scalars) $\mathsf{for}_{L/K}\colon \PervL{X}\to\PervK{X}$ corresponds to a forgetful functor $\GDXfL\to\GDXfK$, where the latter is given by the ones on perverse sheaves on $U$ and $Z$. (Note that the finiteness of $L/K$ is essential here since otherwise the forgetful functor $\mathsf{for}_{L/K}\colon \PervL{X}\to\PervK{X}$ is not well-defined: A perverse sheaf is a complex of sheaves whose cohomologies have in particular have finite-dimensional stalks.)

On the other hand, we have the following statement about extension of scalars.
\begin{lemma}\label{lemma:NCVCext}
	Let $L/K$ be a finite Galois extension.	\\
	For $\cA\in\PervK{U}$, we have an isomorphism
	$$\NCunf(L_U\otimes_{K_U}\cA)\iso L_Z\otimes_{K_Z} \NCunf(\cA).$$
	For $\cB\in\PervK{X}$, we have an isomorphism
	$$\VCunf(L_X\otimes_{K_X}\cB)\iso L_Z\otimes_{K_Z} \VCunf(\cB).$$
	The scalar extension functor $\Phi_{L/K}\colon \PervK{X}\to\PervL{X}$ corresponds to the functor
	\begin{align*}
		\GDXfK&\To\GDXfL\\
		(\cF_U,\cF_Z,u,v)&\longmapsto (L_U\otimes_{K_U}\cF_U,L_Z\otimes_{K_Z}\cF_Z,\id_{L_Z}\otimes u,\id_{L_Z}\otimes v)
	\end{align*}
via Beilinson's equivalence $\Ff$.
\end{lemma}
\begin{proof}
	To prove this, we need an idea of what the gluing data associated to a perverse sheaf are more concretely. The construction of $\NCunf(\cA)$ is performed as follows: One first defines the nearby cycles functor $\RR \varphi_f = i^{-1} \RR j_* \RR \pi_* \pi^{-1}$ (where $i\colon Z\hookrightarrow X$ is the inclusion and $\pi\colon U\times_{\C\setminus\{0\}} \widetilde{\C\setminus\{0\}}\to U$ is the canonical map, with $\widetilde{\C\setminus\{0\}}$ the universal covering; this is, however, not important for what follows). Then one notices that $t$ acts naturally on $\RR \varphi_f(\cA)$ and that there is a decomposition $\RR \varphi_f(\cA)\iso \RR \varphi_f^\mathrm{un}(\cA) \oplus \RR \varphi_f^{\neq 1}(\cA)$, where $\id-t$ is nilpotent on the first and an automorphism on the second summand. Then one sets $\NCunf(\cA)\vcentcolon= \RR \varphi_f^\mathrm{un}(\cA)[-1]$.
	
	It is clear that $\RR \varphi_f$ commutes with extension of scalars (see Proposition~\ref{prop:extIm}). Moreover, the action of $t$ is induced purely topologically, i.e.\ the action of $t$ on $\RR \varphi_f(L_U\otimes_{K_U}\cA)\iso L_Z\otimes_{K_Z} \RR \varphi_f(\cA)$ is induced by the one on $\RR \varphi_f(\cA)$. Hence, the part of $\RR \varphi_f(L_U\otimes_{K_U}\cA)$ on which $\id-t$ is nilpotent will be exactly $L_Z\otimes_{K_Z} \RR \varphi_f^\mathrm{un}(\cA)$. This proves the first statement.
	
	The construction of $\VCunf(\cB)$ is roughly as follows: One first defines the \emph{maximal extension functor} $\Xi_f\colon \Pervk{U}\to \Pervk{X}$ and a complex
	$$j_!j^{-1}\cB\to \Xi_f(j^{-1}\cB)\oplus \cB\to j_*j^{-1}\cB.$$
	Then one defines $\VCunf(\cB)$ as the cohomology of this complex and notes that it is supported on $Z$.
	
	Without going too much into the details of the construction, let us just mention that the definitions of $\Xi_f$ and the morphisms in the above complex are again just topological, i.e.\ using operations that do not depend upon the exact field of coefficients (such as natural morphisms $j_!\to j_*$, inclusions/projections, kernels etc.). Therefore and due to Proposition~\ref{prop:extIm}, the complex associated to $L_X\otimes_{K_X} \cB$ is $$L_X\otimes_{K_X} j_!j^{-1}\cB\to L_X\otimes_{K_X} \big(\Xi_f(j^{-1}\cB)\oplus \cB\big)\to L_X\otimes_{K_X} j_*j^{-1}\cB$$ and finally, since extension of scalars is exact and hence commutes with taking cohomology, we get the second isomorphism of the lemma.
	
	For the last statement, the arguments are similar, remarking that the definition of the maps $u$ and $v$ is topological and can therefore be defined over the smaller field and just ``upgraded'' to $L$.
\end{proof}

Now let $L/K$ be a field extension and $G\vcentcolon=\mathrm{Aut}(L/K)$ (for our purposes, it will be a finite Galois extension with Galois group $G$). There is then an obvious $G$-conjugation on $\GDXfL$ defined by
$$\overline{(\cF_U,\cF_Z,u,v)}^g \vcentcolon= (\overline{\cF_U}^g,\overline{\cF_Z}^g,\overline{u}^g,\overline{v}^g)$$
for any $g\in G$, i.e.\ simply induced by the natural $G$-conjugations on $\PervL{U}$ and $\PervL{Z}$.
\begin{lemma}\label{lemma:NCVCconj}
	Let $L/K$ be a field extension and let $g\in\mathrm{Aut}(L/K)$.\\
	For $\cA\in\PervL{U}$, we have an isomorphism
	$$\NCunf(\overline{\cA}^g)\iso \overline{\NCunf(\cA)}^g.$$
	For $\cB\in\PervL{X}$, we have an isomorphism
	$$\VCunf(\overline{\cB}^g)\iso \overline{\VCunf(\cB)}^g.$$
	Under Beilinson's equivalence $\Ff\colon \PervL{X}\overset{\sim}{\to} \GDXfL$, the natural $G$-conjugations on both categories correspond to each other, and hence a $G$-structure on $\cF\in\PervL{X}$ induces a $G$-structure on $\Ff(\cF)$ and vice versa.
\end{lemma}
\begin{proof}
	Similarly to the proof of Lemma~\ref{lemma:NCVCext}, the first statement follows mainly from the fact that conjugation is compatible with direct and inverse image functors (Lemma~\ref{lemma:compatConj}). Moreover, since conjugation is an autoequivalence, the endomorphism $\id-t$ is nilpotent if and only if $\id-\overline{t}^g=\overline{\id-t}^g$ is.
	
	The second and third statements are again due to the fact that the entire constructions of $\VCunf$, $u$ and $v$ do not depend on the vector space structure and hence are the same if defined before or after applying $g$-conjugation.
\end{proof}

\paragraph{Application to Galois descent of perverse sheaves} We are now ready to prove that the object $\cG$ constructed above is actually a $K$-structure of $\cF$.

\begin{prop}
	The morphism \eqref{eq:latticeComplPerv} is an isomorphism.
\end{prop}
\begin{proof}
	Let $\cF\in\PervL{X}$ be a perverse sheaf on a complex analytic variety $X$ and let $(\varphi_g)_{g\in G}$ be a $G$-structure on it. Let $\cG\in\PervK{X}$ be the invariant $K$-perverse subsheaf (defined analogously as above) with its natural morphism $\cG\to\cF$.
	
	 Since $\cF$ is perverse, it is in particular a complex of sheaves with $\C$-constructible cohomologies. For each of the (finitely many) nontrivial cohomology sheaves $\mathrm{H}^i(\cF)$, there exists a locally finite covering $X=\bigcup_\alpha X^i_\alpha$ by $\C$-analytic subsets\footnote{We follow the terminology in \cite{KS90} here: A subset $Y\subset X$ is called \emph{complex analytic} if for any point $x\in X$ there exists an open neighbourhood $U\subset X$ of $x$ and (finitely many) holomorphic functions $f_1,\ldots,f_n\in\cO_X(U)$ such that $A\cap U=\{f_1=\ldots=f_n=0\}$. Moreover, $Y$ is called $\C$-analytic if $\overline{Y}$ and $\overline{Y}\setminus Y$ are complex analytic subsets of $X$.} on which $\mathrm{H}^i(\cF)$ is locally constant. Moreover, the problem is local (and restriction to an open subset is exact), so we can assume that the set of all $X^i_\alpha$ is finite.
	 
	 If all the $X^i_\alpha$ are of maximal dimension, this means that every cohomology sheaf is locally constant on $X$. By the definition and basic properties of perverse sheaves, $\mathrm{H}^i(\cF)=0$ for $i<-\dim X$ and $\dim \supp \mathrm{H}^{-i}(\cF)\leq i$ for any $i\in \Z$. This implies that $\cF$ is concentrated in cohomological degree $-\dim X$, i.e.\ $\cF\iso \cL[\dim X]$ for some locally constant sheaf $\cL\in\Mod{L_X}$. Hence, the statement follows from Proposition~\ref{prop:GaloisDescentSheaves} and we are done.
	 
	 Now, assume that there exist $X_\alpha^i$ of non-maximal dimension. 
	 Then each of the $X^i_\alpha$ not having maximal dimension is contained in the zero locus of an analytic function that is not identically zero (since $\overline{X^i_\alpha}$ is complex analytic and not equal to the whole space), yielding a finite family of functions $(f_k)_{k\in I}$, $I=\{1,\ldots,m\}$. We can multiply these functions to obtain $f\vcentcolon= f_1\cdot\ldots\cdot f_m$ whose zero locus contains all the $X^i_\alpha$ of non-maximal dimension.
	
	Consider now $Z\vcentcolon= f^{-1}(0)$, $U\vcentcolon= X\setminus Z$ and the inclusion $j\colon U\hookrightarrow X$. By Beilinson's equivalence, the datum of $\cF$ is equivalent to the tuple
	$$(j^{-1}\cF, \VCunf(\cF), u, v)\in \GDXfL.$$
	
	Accordingly, $\cG$ corresponds to a tuple
	$$(j^{-1}\cG,\VCunf(\cG),u_K, v_K)\in \GDXfK$$
	and the morphism induced by $L_X\otimes_{K_X}\cG\to\cF$ via $\Ff$ is nothing but the natural morphism
	$$(L_U\otimes_{K_U}j^{-1}\cG,L_Z\otimes_{K_Z}\VCunf(\cG),\id_{L_Z}\otimes u_K, \id_{L_Z}\otimes v_K)\to (j^{-1}\cF, \VCunf(\cF), u, v)$$
	by Lemma~\ref{lemma:NCVCext}.
	To prove that it is an isomorphism, we need to prove that $L_U\otimes_{K_U}j^{-1}\cG\to j^{-1}\cF$ and $L_Z\otimes_{K_Z}\VCunf(\cG)\to \Phi^{\mathrm{un}}_f(\cF)$ are isomorphisms, where $j^{-1}\cG$ (resp.\ $\VCunf(\cG)$) is the perverse sheaf of invariants of the induced $G$-structure on $j^{-1}\cF$ (resp.\ $\VCunf(\cF)$), since $j^{-1}$ (resp.\ $\VCunf$) is exact and hence commutes with kernels. (Let us note that the unipotent vanishing cycle object $\VCunf(\cF)$ has indeed an induced $G$-structure since, as mentioned in the proof of Lemma~\ref{lemma:NCVCext}, its construction is ``topological'', i.e.\ only uses operations that do not depend on the field and hence are compatible with the $G$-structures involved.)
	
	For the first isomorphism, note that $j^{-1}\cF$ is a complex of sheaves on $U$ whose cohomologies are all locally constant $L_U$-modules of finite rank. Then, with the same arguments as above, $j^{-1}\cF\iso \cL[\dim U]$ for some locally constant sheaf $\cL\in\Mod{L_U}$. Hence, the desired isomorphism follows from Proposition~\ref{prop:GaloisDescentSheaves}.
	
	For the second isomorphism, note that $\VCunf(\cF)$ is a perverse sheaf on the complex analytic variety $Z$ and $\dim Z=\dim X - 1$. Hence, we can apply the same technique (determining, at least locally, a suitable covering of $Z$, choosing a suitable function that vanishes on all the sets of non-maximal dimension and applying Beilinson's equivalence) to this perverse sheaf. We continue this recursively, and the procedure will end if all elements of the covering are of maximal dimension, which will be the case at the latest when $\dim Z=0$, which shows that this inductive procedure terminates.
	This concludes the proof.	
\end{proof}

The statement just proved shows the essential surjectivity part of Galois descent for perverse sheaves. Full faithfulness is inherited from the derived category (Proposition~\ref{prop:ExtDerivedFF}). We have therefore proved the following statement. (We will formulate it in a slightly more general setting than we did in the rest of this work, namely in the context of complex analytic varieties, since this is what we have actually proved.)

\begin{thm}\label{thm:GaloisDescentPerv}
	Let $L/K$ be a finite Galois extension and let $X$ be a complex analytic variety. Then the functor of extension of scalars $\Phi_{L/K}\colon \PervK{X}\to\PervL{X}$ induces an equivalence
	$$\Phi_{L/K}^G\colon \PervK{X}\to\PervL{X}^G.$$
\end{thm}

\vspace{1cm}

\noindent\textsc{Andreas Hohl\\ Université Paris Cité and Sorbonne Université, CNRS, IMJ-PRG, F-75013 Paris, France}\vspace{0.2cm}\\
\textit{Current address:} \textsc{Technische Universität Chemnitz, Fakultät für Mathematik, 09107 Chemnitz, Germany}\\ andreas.hohl@math.tu-chemnitz.de
	
\end{document}